\documentclass[11pt]{article}

\author{Bogdan Alecu \and Vadim Lozin}
\date{}
\title{Understanding lettericity I: a structural hierarchy}
 
\usepackage{graphicx}
\usepackage{tabularx}
\usepackage{hyperref}
\usepackage{amsthm} 
\usepackage{amsmath} 
\usepackage{amsfonts}
\usepackage{amssymb}
\usepackage{tikz}
\usepackage{enumitem}
\usepackage{relsize}
\usepackage{subcaption}
\usetikzlibrary{decorations.pathreplacing}
\tikzstyle{vertex}=[circle,fill=black!100,text=white,inner sep=0.8mm]
\tikzstyle{point}=[circle,fill=black,inner sep=0.1mm]

\DeclareMathOperator{\lett}{let}
\DeclareMathOperator{\av}{Av}
\DeclareMathOperator{\grid}{Grid}
\DeclareMathOperator{\geom}{Geom}

\DeclareMathOperator{\conf}{conf}
\DeclareMathOperator{\free}{Free}

\begin{document}
	\maketitle
	
	\newtheorem{proposition}{Proposition}
	\newtheorem{theorem}{Theorem}
	\newtheorem{lemma}{Lemma}
	\newtheorem{corollary}{Corollary}
	\newtheorem{result}{Result}
	\theoremstyle{definition}
	\newtheorem{definition}{Definition}
	\newtheorem{example}{Example}
	\newtheorem{remark}{Remark}
	\newtheorem{conjecture}{Conjecture}
	\newtheorem{notation}{Notation}
	\newtheorem{problem}{Open problem}
	
	\begin{abstract}
{\em Lettericity} is a graph parameter introduced by Petkov\v sek in \cite{letter-graphs} in order to study well-quasi-orderability under the induced subgraph relation.
In the world of permutations, {\em geometric griddability} was independently introduced in \cite{geometric}, partly as an enumerative tool. 
Despite their independent origins, those two notions share a connection: they highlight very similar structural features in their respective objects. 
The fact that those structural features arose separately on two different occasions makes them very interesting to study in their own right. 

In the present paper, we explore the notion of lettericity through the lens of the ``minimal obstructions'', i.e., minimal classes of graphs of unbounded lettericity,
and identify an infinite collection of such classes. We also discover an intriguing structural hierarchy that arises in the study of lettericity and that of griddability. 
\end{abstract}


\section{Introduction}


In 2002, Petkov\v sek published a paper \cite{letter-graphs} in which he introduced the notion of {\it letter graphs} and a related parameter {\it graph lettericity}. 
This publication was not observed by the research community until 2011, when five papers simultaneously cited
the work of Petkov\v sek. Since then, the notion of letter graphs attracted the attention of many researchers in the field of graph theory. 

In 2013, a group of people working in the area of permutations introduced the notion of {\em geometric griddability}
\cite{geometric}. This notion also attracted considerable attention from researchers in the field. 
However, it was not until 2020 that a close relationship between the two notions - letter graphs and geometric griddability - was discovered and described, first, in one direction \cite{3let},
and then in the other \cite{bddletgg}. Informally, this relationship can be characterised as follows: graph lettericity, restricted to the class of permutation graphs, and geometric griddability
describe the same concept in the language of graphs and permutations, respectively.

Both notions capture important structural properties leading, in particular, to well- (and even better-) quasi-ordering of graphs (by induced subgraphs) and permutations (by pattern containment),
which is a highly desirable but rare property. The fact that the two notions have been introduced independently of each other emphasizes the importance of both of them.
However, our understanding of the two notions remains obscure. The aim of the project initiated by this paper is to fill this lack of knowledge.
 
In the present paper, we address a problem that in the language of permutations can be succinctly stated as follows (the relevant terms will be defined later in Section~\ref{sec:pre}):
\begin{problem} \label{prob:gg}
Characterise geometrically griddable classes of permutations.		
\end{problem}
In the language of graphs, the same problem becomes:
\begin{problem} \label{prob:bddlet}
Characterise graph classes of bounded lettericity.
\end{problem}
Our hope is that this paper will serve as a miniature survey on the problem -- a point from which we may rally our efforts in order to make meaningful progress on it. 
Another goal of this paper is to provide a quick introduction to the problem for those who might be interested in it for its own sake. 
To this end, we attempt to give the minimum amount of background necessary, while trying to paint a picture that is as complete as possible. The paper is organised as follows:
\begin{itemize}
		\item In Section~\ref{sec:pre}, we will introduce the relevant definitions, notation, and terminology, as well as some existing results.
		
		\item Those existing results fit into a (not yet completely understood) hierarchy of structural features; a natural extension of the original problem is to understand the transitions between the different layers of this hierarchy. This discussion is presented in Section~\ref{sec:hierarchy}.
		
		\item In Section~\ref{sec:new-results}, we present our current progress on filling in some of the remaining gaps in our understanding. 
		
		\item Section~\ref{sec:loh} attempts to introduce a more natural terminology for dealing with one of the transitions in the hierarchy.
	
	\end{itemize}

	
\section{Preliminaries}
\label{sec:pre}	

	\subsection{Common preliminaries}
	
	Unless otherwise specified, the graphs in this paper are {\em simple} (that is, undirected, without loops or multiple edges). 
	The vertex set and the edge set of a graph $G$ are denoted $V(G)$ and $E(G)$, respectively. 
	The {\em neighbourhood} of a vertex $x\in V(G)$, denoted $N_G(x)$ (the subscript is omitted when it is clear from context), is the set of vertices adjacent to $x$.
	The {\em degree} of $x$, denoted $\deg(x)$, is the size of its neighbourhood. 
	
	As usual, $K_n,P_n$ and $C_n$ denote a complete graph, a chordless path and a chordless cycle on $n$ vertices, respectively.  
	By $nG$ we denote the disjoint union of $n$ copies of $G$, and $\overline{G}$ denotes the complement of $G$.
	
	The subgraph of $G$ induced by a set $U\subseteq V(G)$ is denoted $G[U]$. 
	If $G$ does not contain an induced subgraph isomorphic to a graph $H$, we say that $G$ is $H$-free, or that $G$ excludes $H$,
	or that $H$ is a forbidden induced subgraph for $G$.  A {\em homogeneous set} is a subset $U$ of $V(G)$ such that $G[U]$ is either complete or edgeless. A {\em class} (or {\em family} or {\em property}) of graphs is a collection of graphs closed under isomorphisms. It is {\em hereditary} if it is closed under taking induced subgraphs. From general theory, any hereditary class $\mathcal X$ can be uniquely described by its set of minimal forbidden induced subgraphs (that is, the graphs not in $\mathcal X$ minimal under the induced subgraph relation).  
	
	We will sometimes work simultaneously with simple graphs, and with certain auxiliary directed graphs. To mitigate ambiguity, we will refer to vertices as ``nodes'' and edges as ``arcs'' in the directed setting.
	
	A {\em partial order} on a set $X$ is a binary relation that is reflexive, antisymmetric and transitive. A set $X$ together with a partial order $\leq$ on $X$ is called a {\em poset}. A partial order $\leq$ on $X$ is said to be {\em total}, or {\em linear}, if any two elements of $X$ are comparable (that is, if for any $x, y \in X$, we have either $x \leq y$ or $y \leq x$). Given a poset $(X, \leq)$, a {\em chain} is a set of pairwise comparable elements (that is, a subset of $X$ totally ordered by $\leq$). An {\em ascending}, respectively {\em descending} chain is a (finite or infinite) sequence $x_1, x_2, \dots$ with $x_1 \leq x_2 \leq \dots$, respectively $x_1 \geq x_2 \geq \dots$. An {\em antichain} is a set of pairwise incomparable elements.
	
	A poset $(X, \leq)$ is well-founded if it contains no infinite strictly descending chain. It is well-quasi-ordered (``wqo'' for short) if it is well-founded, and it contains no infinite antichains.

	\subsection{Letter graphs}

	The notion of letter graphs was introduced in \cite{letter-graphs}. Our terminology differs only superficially from the one used there. We will need some basic notions from the theory of formal languages; rather than defining them separately in the most general setting possible, we will introduce the relevant definitions as we go along, and adapt them to our restricted setting. 
	
	Our starting point is a finite digraph $\mathcal D = (\Sigma, A)$ that we will call the {\em decoding digraph}, or simply {\em decoder}. We call $\Sigma$ a (finite) {\em alphabet}, and we refer to its elements (that is, the vertices of $\mathcal D$) as {\em letters} (or {\em symbols}). Now let $\Sigma^*$ be the set of finite sequences of elements of $\Sigma$. We will refer to them as {\em words} (or {\em strings}) {\em over $\Sigma$}. The main idea is now to construct graphs from words in $\Sigma^*$ by ``decoding'' them using $\mathcal D$. The intuition is that each of the indices $1, \dots, n$ of the word $w = w_1w_2\dots w_n$ corresponds to a vertex, and their adjacency depends (in a straightforward way dictated by the arcs in $\mathcal D$) only on the relative order of the indices and on the symbols appearing at those indices. Formally, we have the following definition:
	
	\begin{definition} \label{def-letter-graph}
		Let $\mathcal D = (\Sigma, A)$ be a decoder, and let $w = w_1w_2\dots w_n \in \Sigma^*$. The {\em letter graph} $G(\mathcal D, w)$ is the finite simple graph defined by 
		\begin{itemize}
			\item $V(G(\mathcal D, w)) = [n]$;
			\item $E(G(\mathcal D, w)) = \{\{i, j\} : (w_{\min(i, j)}, w_{\max(i, j)}) \in A\}$.
		\end{itemize}
		
		The map sending $w$ to $G(\mathcal D, w)$ is called the {\em decoding map}.
	\end{definition}

	Some examples are in order.
	
	\begin{example} \label{ex-lg1}
		
		In Figure~\ref{fig:letex}, we show on the left a decoding digraph $\mathcal D$, and on the right the letter graph $G = G(\mathcal D, acdbad)$. Notice how, for each letter $l \in V(\mathcal D)$, the set $\{i : w_i = l\}$ forms either a clique or an independent set, according to whether the loop $(l, l)$ is in $\mathcal D$ or not. Similarly, notice how, for two letters $l_1, l_2$, the sets $A_s = \{i : w_i = l_s\} (s = 1, 2)$ are complete to each other if $\mathcal D$ contains both arcs $(l_1, l_2)$ and $(l_2, l_1)$, and anticomplete to each other if $\mathcal D$ contains none of the two arcs. Finally, the least trivial situation is when $\mathcal D$ contains exactly one of the arcs $(l_1, l_2)$. For instance, we note in the figure that $\mathcal D$ has the arc $(a, c)$, but not the arc $(c, a)$. This tells us that in $G(\mathcal D, w)$, we connect each $a$ to every $c$ appearing after it, but not to any of the $c$s appearing before it.
		
		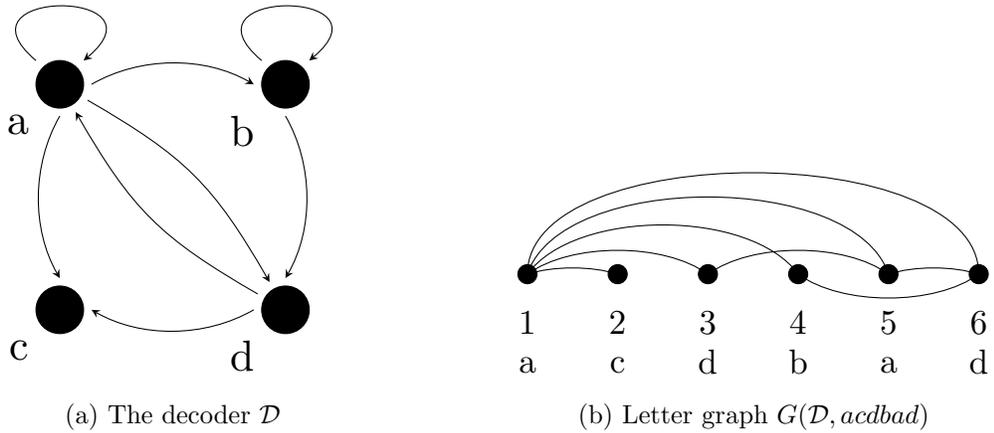
\begin{figure}[ht]
			\centering
			\begin{subfigure}[t]{0.49\linewidth}
				\centering
				\begin{tikzpicture}[scale=1.5, transform shape]
					
					\filldraw
					
					(0,2) circle (6pt) node (a) [below left = 4pt] {a}
					(2,2) circle (6pt) node (b) [below left = 4pt] {b}
					(0,0) circle (6pt) node (c) [below left = 4pt] {c}
					(2,0) circle (6pt) node (d) [below left = 4pt] {d};

					\draw (0cm - 6pt, 2cm + 6pt) edge [out = 140, in = 40, distance = 1cm, -stealth] (0cm + 6pt, 2cm + 6pt);
					\draw (2cm - 6pt, 2cm + 6pt) edge [out = 140, in = 40, distance = 1cm, -stealth] (2cm + 6pt, 2cm + 6pt);
					
					\draw (0cm + 7pt, 2cm - 4pt) edge [out = 330, in = 120, distance = 1cm, -stealth] (2cm - 4pt, 0cm + 7pt);
					\draw (2cm - 7pt, 0cm + 4pt) edge [out = 150, in = 300, distance = 1cm, -stealth] (0cm + 4pt, 2cm - 7pt);
					
					\draw (0cm, 2cm - 8pt) edge [out = 240, in = 120, distance = 0.5cm, -stealth] (0cm, 0cm + 8pt);
					\draw (0cm + 8pt, 2cm) edge [out = 30, in = 150, distance = 0.5cm, -stealth] (2cm - 8pt, 2cm);
					\draw (2cm, 2cm - 8pt) edge [out = 300, in = 60, distance = 0.5cm, -stealth] (2cm, 0cm + 8pt);
					\draw (2cm - 8pt, 0cm) edge [out = 210, in = 330, distance = 0.5cm, -stealth] (0cm + 8pt, 0cm);
					
				\end{tikzpicture}
				\captionsetup{justification=centering}
				\caption{The decoder $\mathcal D$}
				\label{fig:letexa}	
			\end{subfigure}	
			\begin{subfigure}[t]{0.49\linewidth}
				\centering
				\begin{tikzpicture}[scale=1.2, transform shape]
					
					\filldraw
					(1,0) node [below = 22.6pt] {a}
					(2,0) node [below = 22.6pt] {c}
					(3,0) node [below = 20pt] {d}
					(4,0) node [below = 20pt] {b}
					(5,0) node [below = 22.6pt] {a}
					(6,0) node [below = 20pt] {d};
					
					\filldraw	
					(1,0) circle (3pt) node [below = 8pt] {1}
					(2,0) circle (3pt) node [below = 8pt] {2}
					(3,0) circle (3pt) node [below = 8pt] {3}
					(4,0) circle (3pt) node [below = 8pt] {4}
					(5,0) circle (3pt) node [below = 8pt] {5}
					(6,0) circle (3pt) node [below = 8pt] {6};
					
					\draw (1,0) edge [out = 72, in = 108, distance = 1.2cm] (5,0);
					
					\draw (1,0) edge [out = 54, in = 126, distance = 0.9cm] (4,0);
					
					\draw (1,0) edge [out = 18, in = 162, distance = 0.3cm] (2,0);
					
					\draw (1,0) edge [out = 36, in = 144, distance = 0.6cm] (3,0);
					\draw (3,0) edge [out = 36, in = 144, distance = 0.6cm] (5,0);
					\draw (1,0) edge [out = 90, in = 90, distance = 1.5cm] (6,0);
					\draw (5,0) edge [out = 18, in = 162, distance = 0.3cm] (6,0);
					
					\draw (4,0) edge [out = -36, in = -144, distance = 0.6cm] (6,0);	
					
				\end{tikzpicture}
				\captionsetup{justification=centering}
				\caption{Letter graph $G(\mathcal D, acdbad)$}
				\label{fig:letexb}	
			\end{subfigure}
			\caption{A letter graph}
			\label{fig:letex}
		\end{figure}
	\end{example}

	\begin{example} \label{ex-lg2}
		The simplest non-trivial example of a decoder is $\mathcal D = (\{a, b\}, \{(a, b)\})$ (that is, $\mathcal D$ is a digraph with two vertices and a single directed arc between them). Graphs with this decoder are exactly the bipartite chain graphs;\footnote{A bipartite graph is a {\em chain graph} if the vertices in each part are linearly ordered by the inclusion of their neighbourhoods. Equivalently, they are the $2K_2$-free bipartite graphs.} in particular, if $w = abab\dots ab$ is the concatenation of the word $ab$ $n$ times, the graph $G(\mathcal D, w)$ is the prime chain graph on $2n$ vertices (see Figure~\ref{fig-chainex}). The indices in the figure indicate the order in which the vertices appear in $w$.

		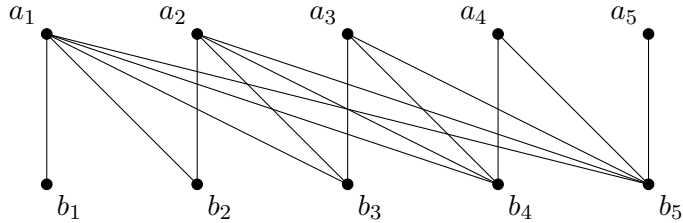
\begin{figure}[ht]
			\begin{center}
				\begin{tikzpicture}[scale=1, transform shape]
					
					\foreach \i in {1,...,5} {
						\filldraw (\i * 2, 0) circle (2pt) node[below right]{$b_{\i}$};
						\filldraw (\i * 2, 2) circle (2pt) node[above left]{$a_{\i}$};
						\foreach \x in {\i,...,5} {
							\draw (\i * 2, 2) -- (\x * 2, 0);
						}
					}

				\end{tikzpicture}
			\end{center}
			\caption{The prime chain graph on 10 vertices}
			\label{fig-chainex}
		\end{figure}
		
	\end{example}
	
	\begin{example} \label{ex-lg3}
		For our final example, we remark that any graph $G$ has a letter graph representation $G \cong G(\mathcal D, w)$, if we put $V(\mathcal D) = V(G)$ and $E(\mathcal D) = \{(u, v), (v, u) : \{u, v\} \in E(G)\}$ ($w$ can be any word containing each letter exactly once). 
	\end{example}
	
	This last example also shows that the question of interest for this notion is not simply ``can we represent a given graph as a letter graph?''. Instead, we want to investigate what happens when we fix a decoder, and consider all graphs representable with that particular decoder. Since, for a given size of the alphabet $\Sigma$, there are only finitely many possible decoders, this is more or less the same as studying what happens when we bound the number of letters. To this end, we have the following definitions:
	
	\begin{definition} \label{lettericity}
		Let $G$ be a graph. The {\em lettericity} $\lett(G)$ of $G$ is the smallest $n \in N$ such that $G$ is isomorphic to a letter graph over a decoder $\mathcal D = (\Sigma, A)$ with $|\Sigma| = n$.
		
		For a decoder $\mathcal D$, we write $\mathcal L_{\mathcal D}$ for the class of graphs representable as letter graphs with that decoder, and call it the {\em class of letter graphs with decoder $\mathcal D$}. For a natural $k$, the {\em class of $k$-letter graphs $\mathcal L_k$} is the (finite) union $\bigcup_{\mathcal D : |V(\mathcal D)| = k} \mathcal L_{\mathcal D}$.
	\end{definition}
	
	\begin{remark} \label{rem-subgraph-implies-subword}
		The classes $\mathcal L_\mathcal D$ (and, as a consequence, $\mathcal L_k$) are hereditary. Indeed, it is easy to check that any induced subgraph $H$ of $G \cong G(\mathcal D, w)$ can be written as $G(\mathcal D, w')$ where $w'$ is obtained from $w$ by deleting the entries not corresponding to vertices of $H$.
	\end{remark}
	
	In \cite{letter-graphs}, Petkov\v sek characterises $k$-letter graphs as follows:
	\begin{theorem}[\cite{letter-graphs}, Proposition~1] \label{thm:k-letter} A graph $G$ is a $k$-letter graph if and only if 
		\begin{itemize}
			\item[$1.$] there is a partition $V_1,V_2,\ldots,V_p$ of $V(G)$ with $p\le k$ such that each $V_i$ is either a clique or an independent set in $G$, and
			\item[$2.$] there is a linear ordering $L$ of $V(G)$ such that for each pair of distinct indices $1\le i,j\le p$, the intersection of $E(G)$ with $V_i\times V_j$ is one of the following four types (where $L$ is considered as a binary relation, i.e., as a set of pairs):
			\begin{itemize}
				\item[{\rm i.}] $L\cap (V_i\times V_j)$;
				\item[{\rm ii.}] $L^{-1}\cap (V_i\times V_j)$;
				\item[{\rm iii.}] $V_i\times V_j$;
				\item[{\rm iv.}] $\varnothing$.
			\end{itemize}  
		\end{itemize}
	\end{theorem}
	
	This characterisation immediately shows that not every class has bounded lettericity, since point 1 above implies lettericity is bounded below by co-chromatic number. 
	
	\bigskip
	
	\bigskip
	
	\bigskip
	
	One of the main reasons the notion of letter graphs is interesting is that $\Sigma^*$ comes with a natural partial order called {\em subword} (or {\em subsequence}) {\em embedding}, that interacts nicely with the induced subgraph partial order:
	
	\begin{definition} \label{def-subword}
		Let $w = w_1w_2\dots w_n$ and $w' = w'_1w'_2\dots w'_{n'}$ be two words over an alphabet $\Sigma$. We say $w$ is a {\em subword} (or {\em subsequence}) of $w'$ (denoted $w \leq w'$) if $n \leq n'$, and there is an increasing injection $\iota : [n] \to [n']$ such that $w_i = w'_{\iota(i)}$ for $i = 1, \dots, n$. $\iota$ is called a {\em subword} or {\em subsequence embedding}.
	\end{definition}
	
	\begin{lemma} \label{lem-subword-implies-subgraph}
		Let $G \cong G(\mathcal D, w)$ and $G' \cong G(\mathcal D, w')$ for some decoder $\mathcal D$ and words $w, w'$. If $w \leq w'$, then $G \leq_i G'$.\footnote{The converse of this lemma is the content of Remark~\ref{rem-subgraph-implies-subword}, so that $w \leq w'$ if and only if $G \leq_i G'$.}
	\end{lemma}	
	\begin{proof}
		$w'$ can be obtained from $w$ by adding letters between the existing letters of $w$ one at a time. The operation of adding a letter does not change the relative order of the original letters, hence when decoding, it corresponds to adding a vertex to the graph and connecting it to some of the original vertices, without changing the adjacency between the original vertices.
	\end{proof}
	
	This simple fact allows us to use order-theoretic results, namely Higman's Lemma \cite{higman-lemma}, on the classes $\mathcal L_\mathcal D$ (and in general, on classes of bounded lettericity).	
	
	\begin{theorem}[Restricted version of Higman's Lemma, \cite{higman-lemma}, Theorems~1.2~and~4.3] \label{thm-higman-restricted}
		The subword relation defined above is a wqo when the alphabet is finite.
	\end{theorem}
	
	\begin{theorem}[\cite{letter-graphs}, Theorem~8] \label{thm-bdd-let-wqo}
		The classes $\mathcal L_k$ are wqo by the induced subgraph relation.
	\end{theorem}
	
	\begin{corollary} \label{cor-bdd-let-wqo}
		Any class of bounded lettericity is wqo.
	\end{corollary}
	
	Theorem~\ref{thm-bdd-let-wqo} makes graph lettericity an important parameter when studying wqo of classes of graphs under the induced subgraph relation, since it provides non-trivial examples of wqo classes of graphs, and it also gives a useful method for proving certain classes are wqo. The theorem also provides an alternative argument that not all classes of graphs have bounded lettericity, since any class containing all cycles (or, indeed, any other infinite antichain) must have unbounded lettericity. Let us construct one explicit example of graphs of high lettericity:
	
	\begin{example} \label{ex-let-matchings}
		Let $n \in \mathbb N$. We have $\lett(nK_2) = n$. Indeed, it is easy to see that $n$ letters are enough to represent the graph (just use one letter per edge). If we had $\lett(nK_2) < n$, then there would be 3 vertices with the same letter, say $a$. Denote their appearances in the word $w$ representing $nK_2$ by $a_1, a_2$ and $a_3$, so that $a_2$ lies between the other two in $w$. We note that no vertex can be adjacent to only the vertex corresponding to $a_2$, which is a contradiction, since every vertex in $nK_2$ has degree 1.	    
	\end{example}
	
	\bigskip
	
	Before moving on, we mention one more result shown in \cite{letter-graphs}.
	
	\begin{theorem}[\cite{letter-graphs}, Theorem~9]
		For each $k$, the class $\mathcal L_k$ is characterised by finitely many minimal forbidden induced subgraphs. 
	\end{theorem}
	\begin{proof}[Sketch of proof]
		Let $S_k$ be the set of minimal forbidden induced subgraphs for $\mathcal L_k$. If $G \in S_k$, then for any vertex $v$ of $G$, $G - v$ is in $\mathcal L_k$. It is not too difficult to see that this implies $G \in \mathcal L_{2k + 1}$. The claim follows, since $\mathcal L_{2k + 1}$ is wqo and $S_k$ is an antichain. 
	\end{proof}

	\subsection{Monotone and geometric griddability}
	\label{let-subsec-grid}
	
	The study of permutations as combinatorial objects is a rich and rapidly developing area of research. Topics of interest include enumerative problems, and well-quasi-orderability. A detailed history of the field, albeit interesting, is outside our scope; we will instead only introduce the notions immediately relevant to us.
	
	For our purposes, a permutation is a linear order on $[n]$ for some $n \in \mathbb N$ -- in other words, a string in which every number in $[n]$ appears exactly once, such as ``41325'' or ``7654321''. We will refer to the characters in the string as {\em digits} or {\em elements}. Permutations come with a natural partial order on them called {\em pattern containment}:
	
	\begin{definition} \label{def-pattern-containment}
		Let $w = w_1w_2\dots w_t$ and $w' = w'_1w'_2\dots w'_{t'}$ be two words in $\mathbb N^*$. We say $w$ is {\em order-isomorphic} to $w'$ if $t = t'$, and for all $1 \leq i, j \leq t$, $w_i \leq w_j$ if and only if $w'_i \leq w'_j$. 
		
		Now let $\sigma$ and $\pi$ be two permutations. We say $\sigma$ is a {\em pattern} of $\pi$ (or $\pi$ {\em contains} $\sigma$ {\em as a pattern}) if $\pi$ contains a subsequence that is order-isomorphic to $\sigma$. If $\pi$ contains no such subsequence, we say $\pi$ {\em avoids} $\sigma$.
	\end{definition}
	
	\begin{example} \label{ex-pattern-containment-1}
		The permutation 2713564 contains 1423 as a pattern. Indeed, the subsequence 2735 is order-isomorphic to 1423.
		
		As another example, the permutations that avoid 21 as a pattern are exactly the increasing permutations 1, 12, 123, 1234, \dots.
	\end{example}
	
	Pattern containment is analogous to the induced subgraph relation, and we can define {\em permutation classes} as sets of permutations closed under (isomorphisms and) pattern containment. By the same general theory as in the case of graphs, any permutation class $\mathcal X$ can be characterised uniquely in terms of its set of minimal avoided patterns $\av(\mathcal X)$, also known as the {\em basis} of $\mathcal X$. 
	
	\begin{remark}
		Now is a good time to point out that in the study of permutations on the one hand, and graphs on the other, completely analogous concepts might have different terminology associated to them. This difference might be subtle -- for instance, graph theorists usually use the word ``hereditary'' to specify when a graph class is closed under taking induced subgraphs, while in the permutation literature, classes are often closed under pattern containment from the definition. We will do our best to avoid any ambiguities caused by this, but the reader should be warned that, when we deem the risk of confusion to be low, we will liberally borrow from one field to refer to concepts from the other, like saying a graph ``avoids'' another (as an induced subgraph). Similarly, we might use more general terminology from standard combinatorial theory, like saying ``minimal obstructions'' (or ``minimal obstacles'') to refer to either minimal forbidden induced subgraphs for a graph class, or to minimal avoided patterns for a permutation class.
	\end{remark}
	
	We may identify a permutation $\pi$ on $[n]$ with its {\em plot}, the set of points $\{(i, \pi(i)) : 1 \leq i \leq n\}$ in the plane. More generally, \cite{geometric} describes a rigorous framework for this geometric perspective on permutations. We do not need the full generality of their framework\footnote{Indeed, \cite{geometric} and to some degree \cite{monotone} present everything with an added level of formalism. This has the benefit of making the geometric theory of permutations and the tools we are about to describe fairly robust, but it does so at the price of brevity. Since our focus is not on permutations themselves, but rather on their relationship to graphs (and, as we will see, to another combinatorial structure capturing some of their order properties), we will take some shortcuts along the way. Our aim here is to give the minimum amount of rigour necessary for developing an intuition in working with those tools; for the reader's peace of mind, we stress that everything we discuss in this subsection could be done carefully and in more detail.}, but the gist of it is as follows: call a set of points in the plane {\em independent} if no two points lie on the same vertical or horizontal line. We may define a permutation as an equivalence class of finite independent sets of points, where two such sets are equivalent if, roughly speaking, we can get from one of them to the other by vertical and horizontal stretching or shrinking. 
	
	As an example, Figure~\ref{fig-examplepermutation} illustrates the plot of 614253 (axes are omitted), which is a representative for its equivalence class. The only thing that matters is the relationship between the vertical and horizontal orderings of the six points. More concretely, if we label the points in increasing order from the bottom to the top, then reading the labels from left to right yields 614253. The full equivalence class consist of exactly the (independent) sets of points with this property. 
	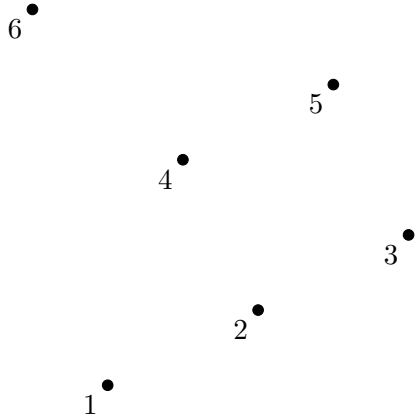
\begin{figure}[ht]
		\begin{center}
			\begin{tikzpicture}[scale=1, transform shape]
				\filldraw 
				(1,6) circle (2pt) node[below left] {6}
				(2,1) circle (2pt) node[below left] {1}
				(3,4) circle (2pt) node[below left] {4}
				(4,2) circle (2pt) node[below left] {2}
				(5,5) circle (2pt) node[below left] {5}
				(6,3) circle (2pt) node[below left] {3};
				
			\end{tikzpicture}
		\end{center}
		\caption{Geometric representation of $\pi = 614253$.}
		\label{fig-examplepermutation}	
	\end{figure}
	
	\bigskip
	
	We will now talk about two tools used to study permutation classes: monotone and geometric griddability. The notion of monotone griddability was developed over several papers, by successively generalising previous notions. Some of the steps that led to the definition that we have today can be found in \cite{atkinson-restricted-perm, atkinson-pwo-closed, murphy-profile}. The definitions we give here are more or less equivalent to the ones introduced in \cite{monotone}.
	
	Let $s, t \in \mathbb N$. An $s \times t$ {\em gridding} $\Gamma$ is a set of $s + 1$ vertical and $t + 1$ horizontal lines in the plane. This partitions the rectangle in the plane defined by the extremal lines into $st$ regions that we will call the {\em cells} of the gridding. The cells are labelled $Z_{ij}$, where the first index counts from left to right, and the second from bottom to top.\footnote{In particular, whenever we use matrices, we will follow the same (non-standard) indexing convention: an $s \times t$ matrix $M = (a_{ij})$ has $s$ columns and $t$ rows; the indices $i, j$ count the entries of $M$ from left to right, and from bottom to top respectively.}
	
	
	\begin{definition}
		Let $\pi$ be a permutation and $M = (\alpha_{ij})$ an $s \times t$ matrix with entries in $\{0, \pm 1\}$.  We say $\pi$ is {\em monotonically griddable by $M$} (or just ``griddable'' for short) if there exists an $s \times t$ gridding $\Gamma$ such that:
		
		\begin{itemize}
			\item If $\alpha_{ij} = 0$, then $\pi \cap Z_{ij} = \varnothing$.
			\item If $\alpha_{ij} = 1$, then $\pi \cap Z_{ij}$ is increasing.
			\item If $\alpha_{ij} = -1$, then $\pi \cap Z_{ij}$ is decreasing.
		\end{itemize}
		
		We say such a $\Gamma$ is a {\em monotone gridding} of $\pi$ by $M$ -- see Figure~\ref{fig-ex-mon-grid} for an example.
		
	\end{definition}
	
	\begin{figure}[ht]
		\begin{center}
			\begin{tikzpicture}[scale=0.8, transform shape]
				\filldraw 
				(1,6) circle (2pt) node[below left] {6}
				(2,1) circle (2pt) node[below left] {1}
				(3,4) circle (2pt) node[below left] {4}
				(4,2) circle (2pt) node[below left] {2}
				(5,5) circle (2pt) node[below left] {5}
				(6,3) circle (2pt) node[below left] {3};
				
				\draw[dashed] (3.5, 7) -- (3.5, 0);
				\draw[dashed] (0, 7) -- (0, 0);
				\draw[dashed] (7, 7) -- (7, 0);
				\draw[dashed] (0, 3.5) -- (7, 3.5);
				\draw[dashed] (0, 0) -- (7, 0);
				\draw[dashed] (0, 7) -- (7, 7);			
			\end{tikzpicture}
		\end{center}
		\caption{A monotone gridding of $614253$ by $\big(\begin{smallmatrix} -1 & -1 \\ 1 & 1 \end{smallmatrix}\big)$.}
		\label{fig-ex-mon-grid}
	\end{figure}
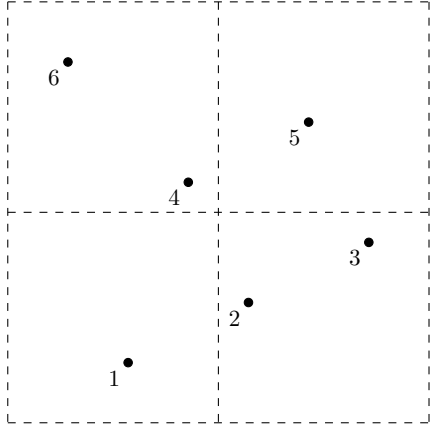
	
	We note that monotone griddability is well-defined: if the plot of a permutation $\pi$ is griddable by some matrix, then so is every other representative for $\pi$ (we may simply stretch the gridding with the plot to get to any independent set of points in the equivalence class of $\pi$). We also note that if $\pi$ is griddable by a matrix $M$, then so is any subpattern $\sigma$ of $\pi$. This motivates the following definitions:
	
	\begin{definition} \label{def-grid-class}
		Let $M$ be a  0/$\pm 1$ matrix. The {\em grid class} of $M$, denoted $\grid(M)$, is the class of permutations monotonically griddable by $M$. A class $\mathcal X$ of permutations is called {\em monotonically griddable} if $\mathcal X \subseteq \grid(M)$ for some fixed 0/$\pm 1$ matrix $M$.        
	\end{definition}
	
	Huczynska and Vatter \cite{monotone} give a characterisation of monotone griddable classes in terms of minimal non-griddable classes. To state it, we first need a definition:
	
	\begin{definition}
		Let $\pi \in S_m$ and $\sigma \in S_n$. We define their {\em direct sum} $\pi \oplus \sigma$ by $$(\pi \oplus \sigma)(i) =
		\begin{cases}
			\pi(i) & \text{if } i \in [m], \\
			\sigma(i - m) + m & \text{if } i \in [m + n] \setminus [m],
		\end{cases} $$
		and similarly their {\em skew sum} $\pi \ominus \sigma$ by $$(\pi \ominus \sigma)(i) =
		\begin{cases}
			\pi(i) + n & \text{if } i \in [m], \\
			\sigma(i - m) & \text{if } i \in [m + n] \setminus [m].
		\end{cases} $$
	\end{definition}
	
	Figure~\ref{fig-direct-skew-sum} illustrates the geometric meaning of the direct and skew sums.
	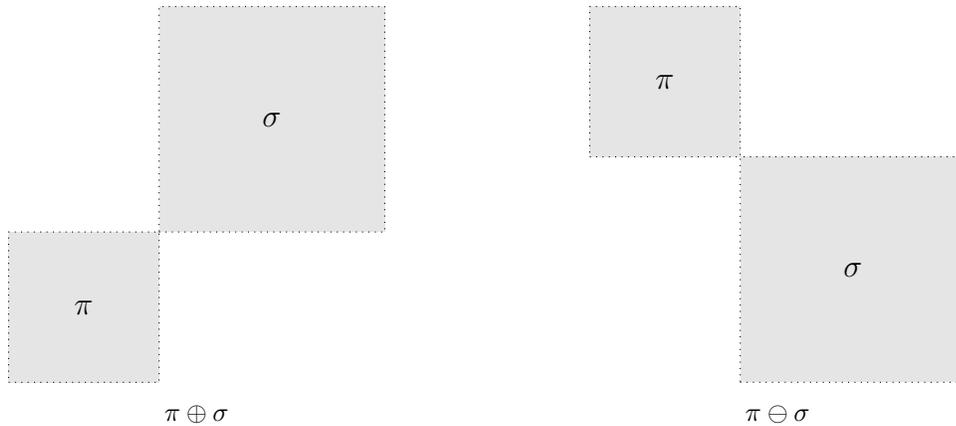
\begin{figure}[ht]
		\centering
		\begin{subfigure}[t]{0.49\linewidth}
			\centering
			\begin{tikzpicture}[scale=1, transform shape]
				
				\fill[gray, fill opacity = 0.2] (0,0) -- (0,2) -- (2,2) -- (2,0) -- cycle;
				\fill[gray, fill opacity = 0.2] (2,2) -- (2,5) -- (5,5) -- (5,2) -- cycle;
				
				\draw[dotted] (0,0) -- (0,2) -- (2,2) -- (2,0) -- cycle;
				\draw[dotted] (2,2) -- (2,5) -- (5,5) -- (5,2) -- cycle;
				
				\draw (1, 1) node{\large $\pi$};
				\draw (3.5, 3.5) node{\large $\sigma$};

			\end{tikzpicture}
			\captionsetup{justification=centering}
			\caption*{$\pi \oplus \sigma$}
		\end{subfigure}	
		\begin{subfigure}[t]{0.49\linewidth}
			\centering
			\begin{tikzpicture}[scale=1, transform shape]
				
				\fill[gray, fill opacity = 0.2] (0,3) -- (0,5) -- (2,5) -- (2,3) -- cycle;
				\fill[gray, fill opacity = 0.2] (2,3) -- (5,3) -- (5,0) -- (2,0) -- cycle;
				
				\draw[dotted] (0,3) -- (0,5) -- (2,5) -- (2,3) -- cycle;
				\draw[dotted] (2,3) -- (5,3) -- (5,0) -- (2,0) -- cycle;
				
				\draw (1, 4) node{\large $\pi$};
				\draw (3.5, 1.5) node{\large $\sigma$};
				
			\end{tikzpicture}
			\captionsetup{justification=centering}
			\caption*{$\pi \ominus \sigma$}	
		\end{subfigure}
		\caption{Direct and skew sum of two permutations}
		\label{fig-direct-skew-sum}
	\end{figure}
	
	\begin{theorem}[\cite{monotone}, Theorem~2.5] \label{thm-huczynska-vatter-griddable-characterisation}
		A permutation class is griddable if and only if it does not contain arbitrarily long direct sums of 21 or skew sums of 12.
	\end{theorem}
	
	In other words, a class $\mathcal X$ of permutations is monotone griddable if and only if it does not contain the class of all (subpatterns of) direct sums of 21 or the class of all (subpatterns of) skew sums of 12.
	
	\bigskip
	
	We now discuss the second, stronger notion of griddability that we mentioned, introduced in \cite{geometric} and called {\em geometric griddability}. The definition is very similar to that of monotone griddability, where we start with a 0/$\pm 1$ matrix $M$ and a gridding whose cells correspond to entries of $M$. However, instead of simply requiring that $\pi$ is monotone in the cells of the gridding, we put the stronger condition that the entries of $\pi$ in each cell lie on one of the diagonals.
	
	\begin{definition}
		Let $\pi$ be a permutation and $M = (\alpha_{ij})$ an $s \times t$ matrix with entries in $\{0, \pm 1\}$. We say $\pi$ is {\em geometrically griddable by $M$} if there exists an $s \times t$ gridding $\Gamma$ such that:
		
		\begin{itemize}
			\item If $\alpha_{ij} = 0$, then $\pi \cap Z_{ij} = \varnothing$.
			\item If $\alpha_{ij} = 1$, then $\pi \cap Z_{ij}$ lies on the main diagonal\footnote{That is, the straight line segment connecting the bottom left corner with the top right one.} of $Z_{ij}$.
			\item If $\alpha_{ij} = -1$, then $\pi \cap Z_{ij}$ lies on the antidiagonal of $Z_{ij}$.
		\end{itemize}
		
		We say such a $\Gamma$ is a {\em geometric gridding} of $\pi$ by $M$ -- see Figure~\ref{fig-ex-geom-grid} for an example.
		
		The union of the diagonals and antidiagonals on which the entries of $\pi$ may lie, subject to suitable normalisations,\footnote{Those normalisations are there just for convenience, and they consist of making each cell correspond to a unit square, with the bottom left corner of the bottom left cell at $(0, 0)$. Since we allow vertical and horizontal stretching, the normalisations do not affect the permutations geometrically griddable by $M$.} is called the {\em standard figure} of $M$.
	\end{definition}
	
	\begin{figure}[ht]
		\begin{center}
			\begin{tikzpicture}[scale=0.8, transform shape]
				\filldraw (0.5,6.5) circle (2pt) node[below left] {6};
				
				\filldraw (1,1) circle (2pt) node[below right] {1};
				
				\filldraw
				(3,4) circle (2pt) node[below left] {4};
				
				\filldraw (5,1.5) circle (2pt) node[below right] {2};
				
				\filldraw (5.5,5) circle (2pt) node[below left] {5};
				
				\filldraw (6.5,3) circle (2pt) node[below right] {3};

				\draw[dashed] (3.5, 7) -- (3.5, 0);
				\draw[dashed] (0, 7) -- (0, 0);
				\draw[dashed] (7, 7) -- (7, 0);
				\draw[dashed] (0, 3.5) -- (7, 3.5);
				\draw[dashed] (0, 0) -- (7, 0);
				\draw[dashed] (0, 7) -- (7, 7);
				
				\draw (0, 0) -- (3.5, 3.5);
				\draw (3.5, 7) -- (7, 3.5);
				\draw (3.5, 0) -- (7, 3.5);
				\draw (0, 7) -- (3.5, 3.5);

			\end{tikzpicture}
		\end{center}
		\caption{A geometric gridding of $614253$ by $\begin{pmatrix}
				-1 & -1 \\
				1 & 1
			\end{pmatrix}$.}
		\label{fig-ex-geom-grid}
	\end{figure}
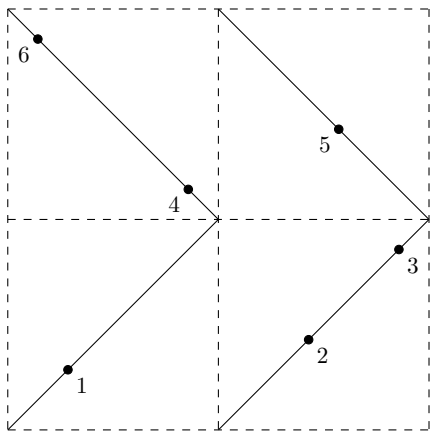
	
	Similarly to (monotone) griddability, we define geometric griddability of classes as follows:
	\begin{definition} \label{def-geom-class}
		Let $M$ be a  0/$\pm 1$ matrix. The {\em geometric grid class} of $M$, denoted $\geom(M)$, is the class of permutations geometrically griddable by $M$. A class $\mathcal X$ of permutations is called {\em geometrically griddable} if $\mathcal X \subseteq \geom(M)$ for some fixed 0/$\pm 1$ matrix $M$.        
	\end{definition}
	
	It is clear from the definition that any permutation geometrically griddable by a matrix is monotonically griddable by that matrix. Concisely, for any 0/$\pm 1$ matrix $M$, $\geom(M) \subseteq \grid(M)$. Is the converse true? As one might expect, the answer is in general negative:
	
	\begin{example}
		Let $\pi = 2413$ and $M =\begin{pmatrix}
			-1 & 1 \\
			1 & -1
		\end{pmatrix}$. Then $\pi \in \grid(M)$, but $\pi \notin \geom(M)$.
		That $\pi \in \grid(M)$ is easy to see -- we can grid it with one element per cell. To see that $\pi \notin \geom(M)$, one can derive a contradiction by first noting (via simple case analysis) that there must be one element per cell, then seeing that, going around in a clockwise cycle starting at say 2, the distance from each element to the centre of the figure should strictly decrease, leaving no place to put 1 (see Figure~\ref{fig-ex-2413}).
	\end{example}
	
	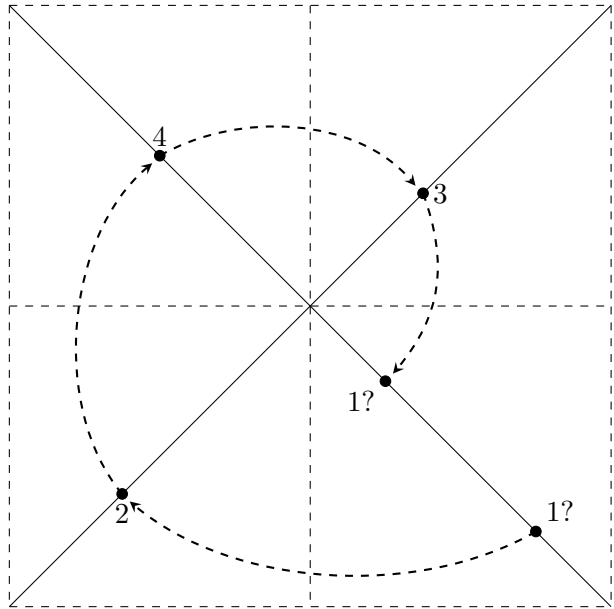
\begin{figure}[ht]
		\centering
		\begin{tikzpicture}[scale=1, transform shape]
			
			\filldraw (1.5, 1.5) circle (2pt) node[below] {2};
			\filldraw (2, 6) circle (2pt) node[above] {4};
			\filldraw (5.5, 5.5) circle (2pt) node[right] {3};
			\filldraw (7, 1) circle (2pt) node[above right] {1?};
			\filldraw (5, 3) circle (2pt) node[below left] {1?};
			
			\draw[dashed] (4, 8) -- (4, 0);
			\draw[dashed] (0, 8) -- (0, 0);
			\draw[dashed] (8, 8) -- (8, 0);
			\draw[dashed] (0, 4) -- (8, 4);
			\draw[dashed] (0, 0) -- (8, 0);
			\draw[dashed] (0, 8) -- (8, 8);
			
			\draw (0, 0) -- (8, 8);
			\draw (0, 8) -- (8, 0);
			
			\draw (1.5, 1.5) edge [out = 130, in = 220, distance = 1.5cm, thick, dashed, -stealth] (1.9, 5.9);
			\draw (2, 6) edge [out = 30, in = 130, distance = 1.2cm, thick, dashed, -stealth] (5.4, 5.6);
			\draw (5.5, 5.5) edge [out = -70, in = 50, distance = 1cm, thick, dashed, -stealth] (5.1, 3.1);
			\draw (7, 1) edge [out = 210, in = -40, distance = 1.8cm, thick, dashed, -stealth] (1.6, 1.4);
			
		\end{tikzpicture}
		\caption{An attempt to geometrically grid 2413.}
		\label{fig-ex-2413}
	\end{figure}
	
	Does it ever happen that $\grid(M)$ and $\geom(M)$ coincide? The answer is ``yes'', and the matrices $M$ for which this is the case are characterised in \cite{geometric}. To state this characterisation, we need the notion of {\em cell graph} of a matrix $M$: for a matrix $M$, the vertices of the cell graph are the non-zero entries of $M$, and two vertices are adjacent if the corresponding entries share a row or a column, and all entries between them are 0 (see Figure~\ref{fig-cell-graph}).
	
	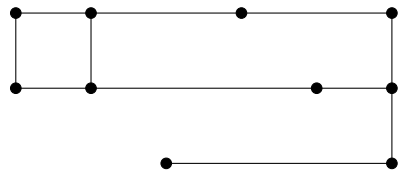
\begin{figure}[ht]
		\centering
		\begin{subfigure}[t]{0.49\linewidth}
			\centering
			\begin{tikzpicture}[scale=1.5, transform shape]
				
				\draw (0, 0) node {$\begin{pmatrix}
						-1 & 1 & 0 & 1 & 0 & -1\\
						1 & 1 & 0 & 0 & -1 & 1\\
						0 & 0 & 1 & 0 & 0 & 1
					\end{pmatrix}$};
				
			\end{tikzpicture}
			\captionsetup{justification=centering}
			\caption{The matrix $M$}	
		\end{subfigure}	
		\begin{subfigure}[t]{0.49\linewidth}
			\centering
			\begin{tikzpicture}[scale=1, transform shape]
				
				\filldraw	
				(0,2) circle (2pt) node[below left] {}
				(0,3) circle (2pt) node[below left] {}
				(1,2) circle (2pt) node[below left] {}
				(1,3) circle (2pt) node[below left] {}
				(2,1) circle (2pt) node[below left] {}
				(3,3) circle (2pt) node[below left] {}
				(4,2) circle (2pt) node[below left] {}
				(5,1) circle (2pt) node[below left] {}
				(5,2) circle (2pt) node[below left] {}
				(5,3) circle (2pt) node[below left] {};
				
				\draw (0, 3) -- (5, 3);
				\draw (0, 2) -- (5, 2);
				\draw (2, 1) -- (5, 1);
				
				\draw (0, 2) -- (0, 3);
				\draw (1, 2) -- (1, 3);
				\draw (5, 1) -- (5, 3);
				
			\end{tikzpicture}
			\captionsetup{justification=centering}
			\caption{Its cell graph.}	
		\end{subfigure}
		\caption{A matrix and its cell graph.}
		\label{fig-cell-graph}
	\end{figure}
	
	The full characterisation says that $\grid(M) = \geom(M)$ if and only if the cell graph of $M$ is a forest, and is an immediate consequence of the following:
	
	\begin{theorem}[\cite{geometric}, Theorem~3.2]\label{thm-forest-grid-geom}
		If the cell graph of $M$ is a forest, then $\grid(M) = \geom(M)$.
	\end{theorem}
	
	\begin{theorem}[\cite{geometric}, Theorem~6.1]\label{thm-gg-wqo}
		Every geometrically griddable class is wqo.
	\end{theorem}
	
	\begin{theorem}[\cite{murphy-profile}, Theorem~2.2]\label{thm-forest-wqo}
		$\grid(M)$ is wqo if and only if the cell graph of $M$ is a forest.
	\end{theorem}
	
	Theorem~\ref{thm-forest-grid-geom} is shown by induction, the main insight being that we have a lot of freedom to move the entries of a permutation lying in a leaf of the cell graph.
	
	Theorems~\ref{thm-gg-wqo}~and~\ref{thm-forest-wqo} together imply the converse of Theorem~\ref{thm-forest-grid-geom}. It is worth noting that Theorem~\ref{thm-gg-wqo} is (besides enumerative results) one of the big reasons geometrically griddable classes are interesting. It is also the first hint suggesting that the notions of geometric griddability and graph lettericity are related (compare with Theorem~\ref{thm-bdd-let-wqo}).
	
	\bigskip
	
	The last item we discuss in this subsection is the proof of Theorem~\ref{thm-gg-wqo}. The proof, which is a straightforward generalisation of work from \cite{vatter-pwo-monotone}, consists of defining a finite alphabet $\Sigma$ depending on $M$, then producing an order preserving, surjective map $\varphi : \Sigma^* \to \geom(M)$. $\Sigma^*$ is wqo by Higman's lemma, and Theorem~\ref{thm-gg-wqo} then immediately follows. To describe the map requires a bit of preparation; we start with a result which, despite being a simple technicality, proves to be very useful. We need a quick definition:
	\begin{definition}
		We say an $s \times t$, 0/$\pm 1$ matrix $M = (\alpha_{ij})$ is a {\em partial multiplication matrix} (``PMM'' for short) if there exist column and row signs $c_1, \dots, c_s, r_1, \dots, r_t \in \{\pm 1\}$ such that $\alpha_{ij}$ is either 0 or the product $c_ir_j$.
	\end{definition}
	
	\begin{proposition}[\cite{geometric}, Proposition~4.2]\label{prop-pmm}
		Every geometric grid class is the geometric grid class of a partial multiplication matrix.
	\end{proposition}
	
	\begin{proof}[Sketch of proof]
		We define a refinement $M^{\times k}$ of a matrix $M$ by replacing each entry with a $k \times k$ matrix. 0s and 1s are replaced with the 0 and identity matrices respectively, and $-1$ is replaced with a matrix with $-1$s on the antidiagonal, and 0s everywhere else. It is not difficult to see that  $\geom(M) = \geom(M^{\times k})$ for any $k$, and that $M^{\times 2}$ is always a partial multiplication matrix.
	\end{proof}
	
	When we work with partial multiplication matrices, the column and row signs will often be represented by arrows to the left of and above the standard figure of the matrix -- as in Figure~\ref{fig-row-col-signs}. Which arrow direction corresponds to which sign is immaterial, as long as it is consistent. The partial multiplication matrix condition says exactly that those arrows can be chosen to always ``agree'' with the diagonals in the cells. This choice of signs also yields a distinguished corner for each cell, as indicated in the figure by the large black dots.
	
	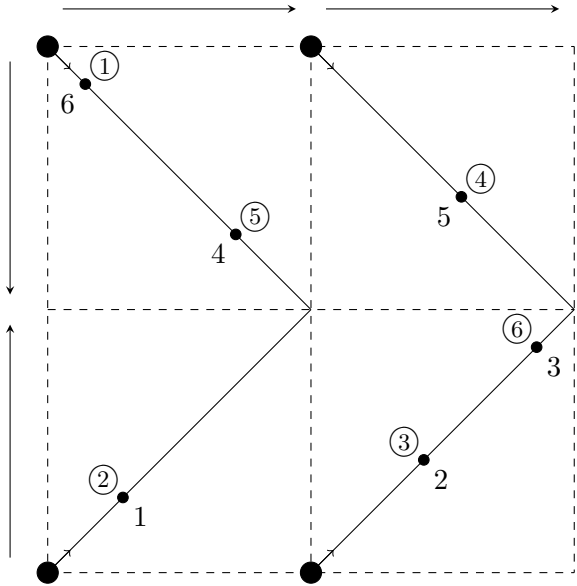
\begin{figure}[ht]
		\centering
		\begin{tikzpicture}[scale=1, transform shape]
			
			\draw[dashed] (3.5, 7) -- (3.5, 0);
			\draw[dashed] (0, 7) -- (0, 0);
			\draw[dashed] (7, 7) -- (7, 0);
			\draw[dashed] (0, 3.5) -- (7, 3.5);
			\draw[dashed] (0, 0) -- (7, 0);
			\draw[dashed] (0, 7) -- (7, 7);
			
			\draw (0, 0) -- (3.5, 3.5);
			\draw (3.5, 7) -- (7, 3.5);
			\draw (3.5, 0) -- (7, 3.5);
			\draw (0, 7) -- (3.5, 3.5);
			
			\draw (-0.5, 0.2) edge [-stealth] (-0.5, 3.3);
			\draw (-0.5, 6.8) edge [-stealth] (-0.5, 3.7);
			\draw (0.2, 7.5) edge [-stealth] (3.3, 7.5);
			\draw (3.7, 7.5) edge [-stealth] (6.8, 7.5);
			
			\filldraw 
			(0,0) circle (4pt)
			(3.5,0) circle (4pt)
			(0,7) circle (4pt)
			(3.5,7) circle (4pt);
			
			\draw (0,0) edge [->] (0.3, 0.3);
			\draw (3.5,0) edge [->] (3.8, 0.3);
			\draw (0,7) edge [->] (0.3, 6.7);
			\draw (3.5,7) edge [->] (3.8, 6.7);
			
			\filldraw (0.5,6.5) circle (2pt) node[below left] {6};
			\node[label={[label distance=-9pt]above right:{\textcircled{\raisebox{-0.55pt}{\smaller{1}}}}}] at (0.5,6.5) {};
			
			\filldraw (1,1) circle (2pt) node[below right] {1};
			\node[label={[label distance=-9pt]above left:{\textcircled{\raisebox{-0.55pt}{\smaller{2}}}}}] at (1,1) {};
			
			\filldraw (2.5,4.5) circle (2pt) node[below left] {4};
			\node[label={[label distance=-9pt]above right:{\textcircled{\raisebox{-0.55pt}{\smaller{5}}}}}] at (2.5,4.5) {};
			
			\filldraw (5,1.5) circle (2pt) node[below right] {2};
			\node[label={[label distance=-9pt]above left:{\textcircled{\raisebox{-0.55pt}{\smaller{3}}}}}] at (5,1.5) {};
			
			\filldraw (5.5,5) circle (2pt) node[below left] {5};
			\node[label={[label distance=-9pt]above right:{\textcircled{\raisebox{-0.55pt}{\smaller{4}}}}}] at (5.5,5) {};
			
			\filldraw (6.5,3) circle (2pt) node[below right] {3};
			\node[label={[label distance=-9pt]above left:{\textcircled{\raisebox{-0.55pt}{\smaller{6}}}}}] at (6.5,3) {};
			
		\end{tikzpicture}
		
		\caption{$\varphi(a_{12}a_{11}a_{21}a_{22}a_{12}a_{21}) = 614253$.}
		\label{fig-row-col-signs}
	\end{figure}
	
	Let us now describe the map $\varphi$. We define the {\em cell alphabet} of the matrix $M = (\alpha_{ij})$ as the set $\Sigma := \{a_{kl} : \alpha_{kl} \neq 0\}$, that is, the letters $a_{kl}$ correspond to non-zero entries of $M$. From a word $w = w_1w_2\dots w_t \in \Sigma^*$, we construct a permutation $\pi \in \geom(M)$ as follows: choose a set $0 < d_1 < d_2 \dots < d_t < 1$ of distances, then for $1 \leq i \leq t$, if $w_i = a_{kl}$, place a point $x_i$ on the diagonal of cell $Z_{kl}$ at infinity-norm distance $d_i$ of the distinguished corner of that cell. Figure~\ref{fig-row-col-signs} illustrates this with $M = \begin{pmatrix}
		-1 & -1 \\
		1 & 1
	\end{pmatrix}$; then $\Sigma = \{a_{11}, a_{12}, a_{21}, a_{22}\}$. If we let $w = a_{12}a_{11}a_{21}a_{22}a_{12}a_{21}$, then choosing $d_i = \frac{i}{7}$ for $i = 1, \dots, 6$ produces the permutation $614253$. The circled numbers are the indices $i$ corresponding to each entry -- one can interpret them as the order in which we insert the elements of the permutation into the picture. 
	
	It is not difficult to see that $\varphi$ is well defined, in that its value does not depend on the choice of distances. Moreover, it is onto and order preserving (\cite{geometric}, Proposition~5.3). $\varphi$ is not injective, since changing the relative order in which we add elements from {\em independent cells} (that is, cells that do not share a row or a column) does not alter the permutation.\footnote{In fact, as shown in \cite{geometric}, for any permutation $\pi$, $\varphi^{-1}(\pi)$ is an equivalence class of words where we are allowed to swap pairs of consecutive letters corresponding to independent cells. Such an equivalence class is called a {\em trace}, and the map $\varphi$ could be made bijective by defining it instead on the so-called {\em trace monoid} $\Sigma^*$ modulo this equivalence relation. Those objects have been studied relatively thoroughly, e.g., in \cite{diekert-traces}. We do not need to concern ourselves with these facts for the time being.}
	
	\subsection{Permutation graphs}
	\label{let-subsec-perm-graphs}
	
	Permutations can be related to graphs via the notion of permutation graphs. To a permutation $\pi$ on $[n]$ we associate its {\em inversion graph} $G_\pi$, whose vertex set is $[n]$ and whose edges are the pairs $\{i, j\}$ that are {\em inverted} by $\pi$, in the sense that $(i - j)(\pi(i) - \pi(j)) < 0$. The class of permutation graphs consists of all graphs that are the inversion graph of some permutation. It has several known characterisations (for instance, as intersection graphs of line segments between two parallel lines, as graphs that are simultaneously comparability and co-comparability \cite{dushnik-perm-comp-cocomp} or as comparability graphs of a poset of order dimension at most two \cite{baker-partial-orders-2}), including a minimal forbidden induced subgraph one \cite{gallai-transitiv}.
	
	The mapping from permutations to their permutation graphs is not injective, since for instance, both $2413$ and $3142$ have $P_4$ as their permutation graph. However, prime graphs with respect to modular decomposition have only two permutation representations that are inverse to each-other (in the function sense) \cite{gallai-transitiv}. We also note that, if $\sigma$ contains $\pi$ as a pattern, then $G_\sigma$ contains $G_\pi$ as an induced subgraph.
	
	\bigskip
	
	The inversion graph of a permutation is particularly easy to read from its geometric plot: two vertices are adjacent if and only if one appears to the bottom right of the other (see Figure~\ref{fig-geom-perm-graph}).
	
	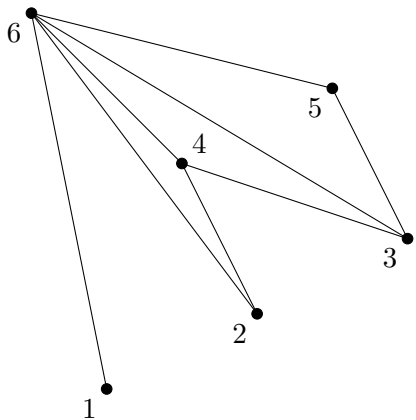
\begin{figure}[ht]
		\centering
		\begin{tikzpicture}[scale=1, transform shape]
			\filldraw 
			(1,6) circle (2pt) node[below left] {6}
			(2,1) circle (2pt) node[below left] {1}
			(3,4) circle (2pt) node[above right] {4}
			(4,2) circle (2pt) node[below left] {2}
			(5,5) circle (2pt) node[below left] {5}
			(6,3) circle (2pt) node[below left] {3};
			
			\draw (1, 6) -- (2, 1);
			\draw (1, 6) -- (3, 4);
			\draw (1, 6) -- (4, 2);
			\draw (1, 6) -- (5, 5);
			\draw (1, 6) -- (6, 3);
			
			\draw (5, 5) -- (6, 3);
			
			\draw (3, 4) -- (4, 2);
			\draw (3, 4) -- (6, 3);
			
		\end{tikzpicture}
		\caption{Permutation graph of 614253.}
		\label{fig-geom-perm-graph}
	\end{figure}
	
	This observation allows us to easily translate properties of (classes of) permutations into properties of the corresponding (classes of) permutation graphs. For instance, any increasing, respectively decreasing sequence in a permutation $\pi$ corresponds to an independent set, respectively a clique in $G_\pi$. Similarly, if $\pi$ admits a gridding by say $M = \begin{pmatrix}1, 1 \end{pmatrix}$, then $G_\pi$ is a bipartite chain graph (the entries of $\pi$ in each of the cells correspond to an independent set in $G_\pi$, and it is not difficult to see the neighbourhoods of the vertices in each part form a chain under inclusion).
	
	In general, if a class $\mathcal X$ is monotonically griddable by a matrix $M$, then there exists a $k$ such that in the corresponding graph class $\mathcal G_\mathcal X$, any graph $G$ admits a partition into at most $k$ bags, where each bag is a clique or an independent set. Those bags correspond to the non-empty cells of the gridding of the permutations by $M$. Moreover, between any two bags we have either no edges, all possible edges, or a bipartite chain graph according to the relative positions of the corresponding cells.
	
%
%
	
	\section{The structural hierarchy}
	\label{sec:hierarchy}
	
	In \cite{3let} and \cite{bddletgg}, we showed that a permutation class $\mathcal X$ is geometrically griddable if and only if the corresponding class $\mathcal G_\mathcal X$ of permutation graphs has bounded lettericity. The ``only if'' direction of the statement is fairly straightforward: for a permutation $\pi$ with a geometric gridding by a partial multiplication matrix $M$, one can carefully construct a decoder $\mathcal D$ over the ``cell alphabet'' of $M$ (whose letters are the non-empty cells of the $M$). A word $w$ such that $G(\mathcal D, w) \cong G_\pi$ is then given by the encoding $\varphi$ from Subsection~\ref{let-subsec-grid}. The ``if'' direction, shown in \cite{bddletgg}, proved to be considerably more tricky from a technical perspective. At the heart of the argument, however, is a simple idea: one first uses Theorem~\ref{thm-huczynska-vatter-griddable-characterisation} to deduce that $\mathcal X$ is {\em monotone} griddable. If one then takes a ``common refinement'' of the information coming from a monotone gridding of $\pi$, together with that coming from a letter graph expression for $G_\pi$, one obtains a geometric gridding of $\pi$ by a matrix whose size only depends on the original monotone gridding matrix.
	
	This equivalence between the two notions is no coincidence. Indeed, one can compare Theorem~\ref{thm:k-letter} with the definition of geometric griddability and of the encoding $\varphi$. In both cases, what we are describing is a partition of the object (graph or permutation) into a bounded number homogeneous bags, together with a linear order which ``plays nicely'' with those bags, in the sense that all of the adjacencies between the bags can be deduced from the linear order using simple rules. One can then ask how much further this structural analogy extends. For instance, is there an equivalent of Theorem~\ref{thm-huczynska-vatter-griddable-characterisation} in the language of graphs? This question was answered by Atminas \cite{star-forests}. A special case of his impressive structural result on star forests describes what happens when we forbid an induced matching and certain complements of an induced matching (analogous to the direct sums of 21s and skew sums of 12s). In that case, we obtain a class in which graphs can be partitioned into a bounded number of homogeneous bags, with $2K_2$-free bipartite graphs between each pair of bags. But from a structural perspective, this does not correspond exactly to monotone griddability by a matrix: indeed, monotone griddability comes with some additional restrictions on how those $2K_2$-free bipartite graph connect to each other. To make all of this clearer, let us explicitly define some graph parameters.
	
	\begin{definition}
		\label{def:proper}
		Let $G = (A, B, E)$ be a bipartite chain graph, and let $\leq_A$ and $\leq_B$ be total orders on $A$ and $B$ respectively. We say $(G, \leq_A, \leq_B)$ is {\em semi-properly ordered} if $\leq_A$ arranges the vertices in $A$ in either increasing or decreasing order of their neighbourhoods, and $\leq_B$ does the same. We say $(G, \leq_A, \leq_B)$ is {\em properly ordered} if, in addition, exactly one of $\leq_A$ and $\leq_B$ arranges its vertices in increasing order. 
	\end{definition}
	
	\begin{definition}
		Let $G$ be a graph. A {\em chain partition} of $G$ is a partition of $V(G)$ into homogeneous sets $A_1, \dots, A_t$ such that, for each $1 \leq i, j \leq t$, the bipartite subgraph consisting of $A_i, A_j$ and the edges between them is $2K_2$-free. We define the {\em chain co-chromatic number} $\gamma(G)$ of $G$ as the smallest natural number $t$ such that $V(G)$ admits a chain partition into $t$ bags. 
		
		A {\em locally semi-consistent chain partition} of $G$ is a partition of $V(G)$ into totally ordered homogeneous sets $A_1, \dots, A_t$ such that, for each $1 \leq i, j \leq t$, the subgraph consisting of $A_i, A_j$ and the edges between them is a semi-properly ordered chain graph in the sense of Definition~\ref{def:proper}. We define the {\em locally semi-consistent chain co-chromatic number} $\sigma(G)$ of $G$ as the smallest natural number $t$ such that $V(G)$ admits a locally semi-consistent chain partition into $t$ bags. 
		A class of graphs is {\em locally semi-consistent} if $\sigma$ is bounded in it.
		
		A {\em locally consistent chain partition} of $G$ is a partition of $V(G)$ into totally ordered homogeneous sets $A_1, \dots, A_t$ such that, for each $1 \leq i, j \leq t$, the subgraph consisting of $A_i, A_j$ and the edges between them is a properly ordered chain graph in the sense of Definition~\ref{def:proper}. We define the {\em locally consistent chain co-chromatic number} $\lambda(G)$ of $G$ as the smallest natural number $t$ such that $V(G)$ admits a locally consistent chain partition into $t$ bags. 
		A class of graphs is {\em locally consistent} if $\lambda$ is bounded in it.
	\end{definition}
	
	The three conditions described above are increasingly stronger, and the lettericity partition described in Theorem~\ref{thm:k-letter} -- we may call it a {\em globally consistent chain partition} --  is even stronger. We then get a chain of implications: bounded lettericity implies bounded $\lambda$ implies bounded $\sigma$ implies bounded $\gamma$. In the next section, we will show that all of those gaps are proper. Intuitively, all of them except for the one between $\gamma$ and $\sigma$ have clear analogues in the world of permutations, as we illustrate in Figure~\ref{fig-stuctural-hierarchy}.
	
	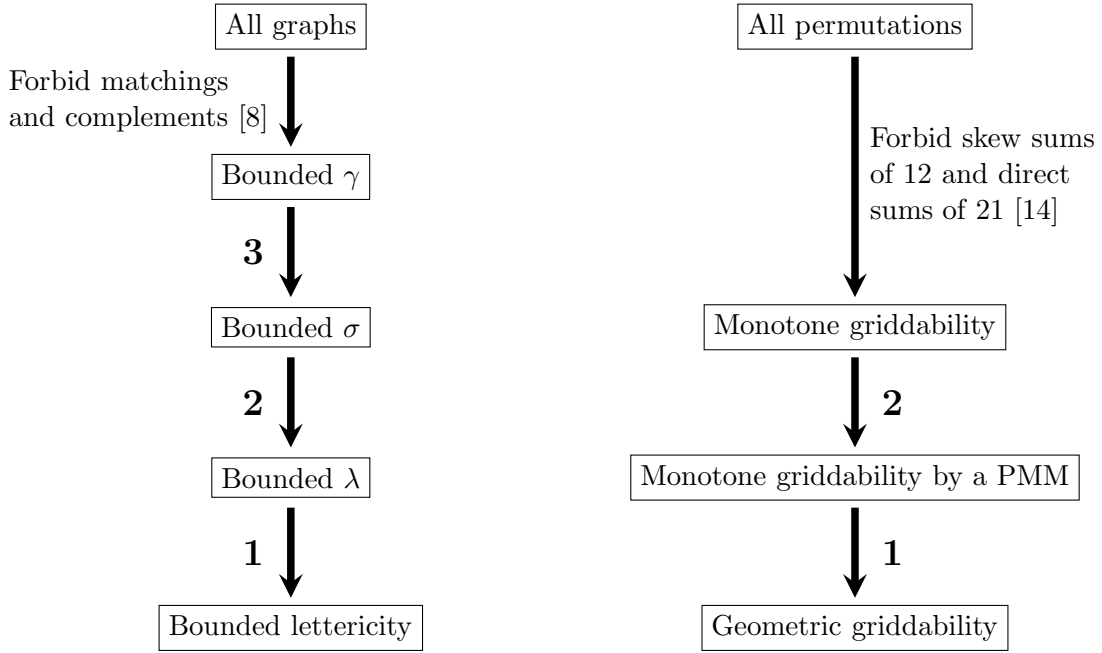
\begin{figure}[ht]
		\centering
		\begin{tikzpicture}
			
			\draw (7.5, 6) node[draw] {All permutations};
			
			\draw[-stealth, line width = 1mm] (7.5, 5.6) -- (7.5, 2.4);
			
			\draw (7.5, 2) node[draw] {Monotone griddability};
			
			\draw[-stealth, line width = 1mm] (7.5, 1.6) -- (7.5, 0.4);
			
			\draw (7.5, 0) node[draw] {Monotone griddability by a PMM};
			
			\draw[-stealth, line width = 1mm] (7.5, -0.4) -- (7.5, -1.6);
			
			\draw (7.5, -2) node[draw] {Geometric griddability};
			
			\draw (0, 6) node[draw] {All graphs};
			
			\draw (9.2, 4) node[text width = 3cm] {Forbid skew sums of 12 and direct sums of 21 \cite{monotone}};
			
			\draw (8, 1) node {\Large {\textbf{2}}};
			
			\draw (8, -1) node {\Large {\textbf{1}}};

			\draw (0, -2) node[draw] {Bounded lettericity};
			
			\draw[-stealth, line width = 1mm] (0, -0.4) -- (0, -1.6);
			
			\draw (0, 0) node[draw] {Bounded $\lambda$};
			
			\draw[-stealth, line width = 1mm] (0, 5.6) -- (0, 4.4);
			
			\draw (-2, 5) node[text width = 3.5cm] {Forbid matchings and complements \cite{star-forests}};
			
			\draw (0, 4) node[draw] {Bounded $\gamma$};
			
			\draw (0, 2) node[draw, align = center] {Bounded $\sigma$};
			
			\draw[-stealth, line width = 1mm] (0, 1.6) -- (0, 0.4);
			
			\draw (-0.5, 1) node {\Large {\textbf{2}}};
			
			\draw (-0.5, -1) node {\Large {\textbf{1}}};
			
			\draw[-stealth, line width = 1mm] (0, 3.6) -- (0, 2.4);
			
			\draw (-0.5, 3) node {\Large {\textbf{3}}};
			
		\end{tikzpicture}
		
		\caption{A hierarchy of structure}
		\label{fig-stuctural-hierarchy}
	\end{figure}
	
	This hierarchy suggests several natural questions closely related (although not directly equivalent) to Problems~\ref{prob:gg}~and~\ref{prob:bddlet}. Those questions concern the successive transitions between the structural layers. The topmost transition on the graph side has been solved by Atminas; the solution can be summarised in the following theorem, which is a straightforward consequence of the results from \cite{star-forests}:
	
	\begin{theorem}
		The only minimal classes of unbounded $\gamma$ are the class $\mathcal M$ of induced matchings, as well as some complements of this class.\footnote{We will describe more carefully what we mean by ``complements'' in the next section.}
	\end{theorem}
	
	The other transitions are not yet completely understood (although we should mention that Huczynska and Vatter's characterisation from \cite{monotone} implies that, when restricted to permutation graphs, bounded $\gamma$ is the same as bounded $\sigma$). The use of the same number next to the arrows in the figure indicates that, from a structural perspective, the transitions at play are more or less identical. In other words, any result from one of the two terminologies should in principle be easily translatable into the other; any difficulties arising would be technical, and not conceptual. Explicitly, in the graph terminology, the natural questions to ask are:
	
	\begin{problem}\label{gamma-sigma}
		Among classes of bounded $\gamma$, what are the obstacles to bounded $\sigma$?
	\end{problem}

	\begin{problem}\label{sigma-lambda}
		Among classes of bounded $\sigma$, what are the obstacles to bounded $\lambda$?
	\end{problem}

	\begin{problem}\label{lambda-let}
		Among classes of bounded $\lambda$, what are the obstacles to bounded lettericity?
	\end{problem}
	
	We remark that answers to Problems~\ref{sigma-lambda}~and~\ref{lambda-let} would, in all likelihood, provide a satisfactory answer to Problem~\ref{prob:gg}. However, while the solution for Problem~\ref{prob:bddlet} will have significant overlap with the three problems above, we might need some additional adjustments to obtain it; we will return to this in the next section.  
	
	
	\section{Our current progress}
	\label{sec:new-results}
	
	Our results on the topic are somewhat spread across the various pieces of the hierarchy presented in the previous section. We will strive to group them together in a way that makes sense. A good point to start, in view of the formulation of Problem~\ref{prob:bddlet}, is a sanity check that bounded lettericity {\em can} be characterised by minimal classes. To this end, we give the following proposition; its proof should appear unsurprising to anyone who has worked with wqo.
	
	\begin{proposition} \label{prop-let-min-class}
		Let $\mathcal{X}$ be a hereditary class of unbounded lettericity. Then there exists a hereditary class $\mathcal{X'} \subseteq \mathcal{X}$ of unbounded lettericity such that the class $\mathcal{X'} \cap \free(G)$ has bounded lettericity for any $G \in \mathcal{X'}$. 
		
	\end{proposition}
	
	\begin{proof}
		
		If $\mathcal{X}$ is minimal of unbounded lettericity, we are done. Otherwise, $\mathcal{X}$ contains a graph $G$ such that $\mathcal{X} \cap \free(G)$ still has unbounded lettericity. Pick a graph $G_0 \in \mathcal{X}_0 := \mathcal{X}$ with this property that has the minimum possible number of vertices, and put $\mathcal{X}_1 := \mathcal{X}_0 \cap \free(G_0)$. Repeat the process for as long as possible, putting $\mathcal{X}_{k+1} := \mathcal{X}_k \cap \free(G_k)$, where $G_k \in \mathcal{X}_k$ is a minimum graph such that $\mathcal{X}_{k+1}$ has unbounded lettericity. There are two cases:
		
		\begin{itemize}
			
			\item The process stops at some $k$. This means we have found a subclass $\mathcal{X}_k = \mathcal{X} \cap \free(G_0, ..., G_{k-1})$ of unbounded lettericity such that forbidding any further $G \in \mathcal{X}_k$ yields a class of bounded lettericity, and $\mathcal X_k$ is the minimal class we were looking for. 
			
			\item Otherwise, the process goes on forever, and we get an infinite strictly descending chain $\mathcal{X}_0 \supsetneq \mathcal{X}_1 \supsetneq \dots$ of classes, all of which have unbounded lettericity. Let $\mathcal{X}_{\lim}$ be their intersection, that is, $\mathcal{X}_{\lim} = \mathcal{X} \cap \free(G_0, G_1, ...)$. 
			
			Note that this limit class cannot have bounded lettericity. Indeed, suppose $\mathcal{X}_{\lim}$ has lettericity at most $t$. The $G_i$s are by construction incomparable via the induced subgraph relation (since at each stage, they are chosen to be minimal), and from \cite{letter-graphs}, they all have lettericity bounded above by $2t + 1$ (since for any $v \in G_i, i \in \mathbb{N}, G_i - v$ is in $\mathcal{X}_{\lim}$). This is a contradiction, since classes of bounded lettericity are well-quasi-ordered (Theorem~\ref{thm-bdd-let-wqo}).
			
			Moreover, $\mathcal{X}_{\lim}$ is minimal of unbounded lettericity. To see this, note first that, by construction, $|G_i| \leq |G_j|$ for $i \leq j$, and $|G_n| \to \infty$ as $n \to \infty$ (since there are only finitely many graphs of a given size). Suppose we can forbid $G \in \mathcal{X}_{\lim}$ with $|G| = k$, and we are still left with a class of unbounded lettericity. But then, by construction, $G$ would have appeared in the sequence $(G_i)$ before any graphs of size at least $k + 1$, contradicting $G \in \mathcal{X}_{\lim} = \mathcal{X} \cap \free(G_1, G_2, \dots)$.
		\end{itemize} 
	
	\end{proof}

	This shows, in particular, that there exists some list $L$ of classes with the property that any class $\mathcal X$ has bounded lettericity if and only if it contains no class from $L$. Describing $L$ is the most straightforward way to answer Problem~\ref{prob:bddlet}. As of yet, we have not checked whether analogues to this proposition hold for $\lambda$ and $\sigma$ (for $\gamma$, this is a direct consequence of the results from \cite{star-forests}). We leave those questions here as conjectures, since our evidence so far supports this formulation, and we have no reason to believe the answer is negative:
	
	\begin{conjecture}
		Any class of unbounded $\sigma$ contains a minimal class of unbounded $\sigma$.
	\end{conjecture}

	\begin{conjecture}
		Any class of unbounded $\lambda$ contains a minimal class of unbounded $\lambda$.
	\end{conjecture}
	
	\bigskip
	
	The next point we want to make is that the solutions to our problems present certain symmetries -- the so-called ``complementations'' that we mentioned in passing. The easiest way to see those symmetries is through the example of induced matchings. As we have already seen in Example~\ref{ex-let-matchings}, induced matchings are a minimal class of unbounded lettericity. By symmetry, one sees that the complements of matchings are also an example. In fact, one can produce more examples by taking various complements that correspond intuitively to making changes in the decoder. For instance, let us consider a letter graph expression of $nK_2$ with a fixed bipartition. By at most doubling the number of letters, we can obtain an expression that uses each letter in only one part of the bipartition; making appropriate changes to the decoder, the same word will then express, say, a split graph with a co-matching between the clique and the independent set. From here, it is easy to show that the class of graphs consisting of a clique co-matched to an independent set is another minimal class of unbounded lettericity. In the cases we study, it is usually safe to assume that, if a class is minimal of unbounded lettericity, so are the classes obtained from it by complementing homogeneous sets, or edges between homogeneous sets. We will not dwell on the details, since they are immaterial to our discussion. In the remainder of the paper, we will use the word ``complements'' to refer to all of the classes obtained from a given class via such complementations; it should be clear from context what we refer to (or at the very least, it should be clear how one would be able to obtain a precise description of what we refer to). The reader who wants to see an illustrated example of those complements is invited to consult, for instance, \cite{star-forests}.
	
	\bigskip
	
	Let us now summarise what we know about the individual transitions in the graph setting. As shown in \cite{star-forests}, the first transition, between all classes and classes of bounded $\gamma$, happens when we forbid matchings and complements; those are minimal classes of unbounded $\gamma$, and also minimal classes of unbounded lettericity. We will examine the other three transitions separately. In each case, we will first show that the transition is non-trivial, by providing explicit examples. Second, we will provide a discussion about what we believe all obstacles should look like, and how one would go about finding them, with partial results and conjectures as appropriate (that is, we will talk about how one may tackle Problems~\ref{gamma-sigma},~\ref{sigma-lambda}~and~\ref{lambda-let}). Finally, we will relate the transitions to Problem~\ref{prob:bddlet} by discussing what the obstacles to bounded lettericity are in each layer of the hierarchy. For convenience, we will start at the bottom of the hierarchy.
	
	\subsection{Between bounded $\lambda$ and bounded lettericity}
	
	The first transition that we examine in detail is the one that kicked off this entire research direction -- in the terminology of permutations, it is the transition between monotone griddability by a PMM to geometric griddability. As claimed, we will start by giving some examples lying ``between'' $\lambda$ and lettericity, that is, we will produce families of bounded $\lambda$, but unbounded lettericity. Several of those examples occurred ``in the wild'', while trying to understand lettericity in various restricted contexts (for instance, within a certain class of graphs). It is interesting to note that the same very similar examples kept recurring. For instance, in \cite{alecu-2p3}, we identify minimal classes of unbounded lettericity among one-sided $2P_3$-free bipartite graphs. Our list consists of induced matchings, their bipartite complements, and a third class that we dubbed ``double-chain graphs'': the universal graphs for it look like prime chain graphs, except that each vertical edge is inflated to a $2K_2$. An identical example up to some complementations arises independently in Ferguson and Vatter's study of classes with only finitely many prime graphs \cite{primelet}.\footnote{This same example will appear once more in a somewhat unexpected place a bit later in this paper; the reader may wish to keep an eye out for it!} Those examples are instances of a ``cycles-in-a-chain'' construction\footnote{Although ``chain graphs-in-a-cycle'' is perhaps also an adequate way to describe it.} illustrated in Figure~\ref{fig-ex-cc}.
	
	To put this example into context, we start by describing a somewhat simplified setting in which it lives. This setting more or less corresponds to the monotone grid class of PMM whose cell graph is a cycle. It is also straightforward generalisation of the ``nice'' graphs studied in \cite{3let}. That $\lambda$ is bounded in those graphs is an immediate consequence of the definition.

	\begin{definition}
		Let $k \geq 3$. A \emph{$k$-chain circuit} ($k$-CC for short) is a graph whose vertex set can be partitioned into $k$ independent sets (or ``bags'') $A_1, \dots, A_{k}$ with indices modulo $k$, such that:
		\begin{itemize}
			\item the only edges appear between consecutive $A_i$;
			\item the edges between consecutive $A_i$ induce chain graphs;
			\item if we order the vertices from $A_i$ in decreasing order with respect to their neighbourhoods in $A_{i+1}$, then that order is increasing with respect to the neighbourhoods in $A_{i-1}$.
		\end{itemize}
		We will refer to $\{A_i : 1 \leq i \leq k\}$ as the \emph{chain circuit partition}, and we will assume the vertices in each bag are ordered as described above. If $v$ comes before $w$ in the order, we write $v \preceq w$, and we say $v$ is ``to the left'' of $w$ (and $w$ is ``to the right'' of $v$).
		Given a chain circuit, we can also define its \emph{chain circuit complement} by complementing all edges between consecutive bags. It is clear that the graph thus obtained is still a chain circuit.
		
	\end{definition}
	
	Amongst $k$--chain circuits, we look at two subclasses of graphs. First, we introduce some notation:
	
	\begin{notation} \label{not-cycles}
		Let $C_{k,l}$ denote the chain circuit obtained by taking a union of $l$ $k$--cycles, with the $j$-th cycle labelled $v_{1,j}, v_{2,j}, ..., v_{k,j}$ (the first index is modulo $k$), and by adding edges between vertices $v_{i,m}$ and $v_{i+1, n}$ whenever $m < n$. See Figure~\ref{fig-ex-cc} for two representations of the graph $C_{4, 4}$.
	\end{notation}
	
	\begin{figure}[ht]
		\centering
		\begin{subfigure}[ht]{0.49\linewidth}
			\centering
			\begin{tikzpicture}[scale = 1.5, transform shape]
				
				\filldraw (-15pt, -15pt) circle (1pt) node[below left] {};
				\filldraw (-5pt, -5pt) circle (1pt) node[below left] {};
				\filldraw (5pt, 5pt) circle (1pt) node[below left] {};
				\filldraw (15pt, 15pt) circle (1pt) node[below left] {};
				
				\filldraw[rotate around ={90:(1.25, 1.25)}] (-15pt, -15pt) circle (1pt) node[below left] {};
				\filldraw[rotate around ={90:(1.25, 1.25)}] (-5pt, -5pt) circle (1pt) node[below left] {};
				\filldraw[rotate around ={90:(1.25, 1.25)}] (5pt, 5pt) circle (1pt) node[below left] {};
				\filldraw[rotate around ={90:(1.25, 1.25)}] (15pt, 15pt) circle (1pt) node[below left] {};
				
				\filldraw[rotate around ={180:(1.25, 1.25)}] (-15pt, -15pt) circle (1pt) node[below left] {};
				\filldraw[rotate around ={180:(1.25, 1.25)}] (-5pt, -5pt) circle (1pt) node[below left] {};
				\filldraw[rotate around ={180:(1.25, 1.25)}] (5pt, 5pt) circle (1pt) node[below left] {};
				\filldraw[rotate around ={180:(1.25, 1.25)}] (15pt, 15pt) circle (1pt) node[below left] {};
				
				\filldraw[rotate around ={270:(1.25, 1.25)}] (-15pt, -15pt) circle (1pt) node[below left] {};
				\filldraw[rotate around ={270:(1.25, 1.25)}] (-5pt, -5pt) circle (1pt) node[below left] {};
				\filldraw[rotate around ={270:(1.25, 1.25)}] (5pt, 5pt) circle (1pt) node[below left] {};
				\filldraw[rotate around ={270:(1.25, 1.25)}] (15pt, 15pt) circle (1pt) node[below left] {};
				
				\draw (-15pt, 2.5cm + 15pt) -- (-15pt, -15pt);
				\draw (-15pt, 2.5cm + 15pt) -- (-5pt, -5pt);
				\draw (-15pt, 2.5cm + 15pt) -- (5pt, 5pt);
				\draw (-15pt, 2.5cm + 15pt) -- (15pt, 15pt);
				
				\draw (-5pt, 2.5cm + 5pt) -- (-5pt, -5pt);
				\draw (-5pt, 2.5cm + 5pt) -- (5pt, 5pt);
				\draw (-5pt, 2.5cm + 5pt) -- (15pt, 15pt);
				
				\draw (5pt, 2.5cm - 5pt) -- (5pt, 5pt);
				\draw (5pt, 2.5cm - 5pt) -- (15pt, 15pt);
				
				\draw (15pt, 2.5cm - 15pt) -- (15pt, 15pt);

				\draw[rotate around ={90:(1.25, 1.25)}] (-15pt, 2.5cm + 15pt) -- (-15pt, -15pt);
				\draw[rotate around ={90:(1.25, 1.25)}] (-15pt, 2.5cm + 15pt) -- (-5pt, -5pt);
				\draw[rotate around ={90:(1.25, 1.25)}] (-15pt, 2.5cm + 15pt) -- (5pt, 5pt);
				\draw[rotate around ={90:(1.25, 1.25)}] (-15pt, 2.5cm + 15pt) -- (15pt, 15pt);
				
				\draw[rotate around ={90:(1.25, 1.25)}] (-5pt, 2.5cm + 5pt) -- (-5pt, -5pt);
				\draw[rotate around ={90:(1.25, 1.25)}] (-5pt, 2.5cm + 5pt) -- (5pt, 5pt);
				\draw[rotate around ={90:(1.25, 1.25)}] (-5pt, 2.5cm + 5pt) -- (15pt, 15pt);
				
				\draw[rotate around ={90:(1.25, 1.25)}] (5pt, 2.5cm - 5pt) -- (5pt, 5pt);
				\draw[rotate around ={90:(1.25, 1.25)}] (5pt, 2.5cm - 5pt) -- (15pt, 15pt);
				
				\draw[rotate around ={90:(1.25, 1.25)}] (15pt, 2.5cm - 15pt) -- (15pt, 15pt);

				\draw[rotate around ={180:(1.25, 1.25)}] (-15pt, 2.5cm + 15pt) -- (-15pt, -15pt);
				\draw[rotate around ={180:(1.25, 1.25)}] (-15pt, 2.5cm + 15pt) -- (-5pt, -5pt);
				\draw[rotate around ={180:(1.25, 1.25)}] (-15pt, 2.5cm + 15pt) -- (5pt, 5pt);
				\draw[rotate around ={180:(1.25, 1.25)}] (-15pt, 2.5cm + 15pt) -- (15pt, 15pt);
				
				\draw[rotate around ={180:(1.25, 1.25)}] (-5pt, 2.5cm + 5pt) -- (-5pt, -5pt);
				\draw[rotate around ={180:(1.25, 1.25)}] (-5pt, 2.5cm + 5pt) -- (5pt, 5pt);
				\draw[rotate around ={180:(1.25, 1.25)}] (-5pt, 2.5cm + 5pt) -- (15pt, 15pt);
				
				\draw[rotate around ={180:(1.25, 1.25)}] (5pt, 2.5cm - 5pt) -- (5pt, 5pt);
				\draw[rotate around ={180:(1.25, 1.25)}] (5pt, 2.5cm - 5pt) -- (15pt, 15pt);
				
				\draw[rotate around ={180:(1.25, 1.25)}] (15pt, 2.5cm - 15pt) -- (15pt, 15pt);

				\draw[rotate around ={270:(1.25, 1.25)}] (-15pt, 2.5cm + 15pt) -- (-15pt, -15pt);
				\draw[rotate around ={270:(1.25, 1.25)}] (-15pt, 2.5cm + 15pt) -- (-5pt, -5pt);
				\draw[rotate around ={270:(1.25, 1.25)}] (-15pt, 2.5cm + 15pt) -- (5pt, 5pt);
				\draw[rotate around ={270:(1.25, 1.25)}] (-15pt, 2.5cm + 15pt) -- (15pt, 15pt);
				
				\draw[rotate around ={270:(1.25, 1.25)}] (-5pt, 2.5cm + 5pt) -- (-5pt, -5pt);
				\draw[rotate around ={270:(1.25, 1.25)}] (-5pt, 2.5cm + 5pt) -- (5pt, 5pt);
				\draw[rotate around ={270:(1.25, 1.25)}] (-5pt, 2.5cm + 5pt) -- (15pt, 15pt);
				
				\draw[rotate around ={270:(1.25, 1.25)}] (5pt, 2.5cm - 5pt) -- (5pt, 5pt);
				\draw[rotate around ={270:(1.25, 1.25)}] (5pt, 2.5cm - 5pt) -- (15pt, 15pt);
				
				\draw[rotate around ={270:(1.25, 1.25)}] (15pt, 2.5cm - 15pt) -- (15pt, 15pt);

			\end{tikzpicture}	
			\captionsetup{justification=centering}
			\caption{Bags in a cycle \\ with chain graphs between them.}
		\end{subfigure}
		\begin{subfigure}[ht]{0.49\linewidth}
			\centering
			\begin{tikzpicture}[scale = 0.5, transform shape]
				
				\clip (-4,-2.5) rectangle ++(7.1,14);
				
				\foreach \i in {0,...,3} {
					\foreach \x in {0,...,3} {
						\filldraw (\i, 3 * \x) circle (2pt) node{};
					}
				}
				
				\foreach \i in {-1,...,3} {
					\foreach \x in {0,...,3} {
						\foreach \y in {\x, ..., 3} {
							\draw (\y, 3 * \i) -- (\x, 3 * \i + 3);
						}
					}
				}
				
				\filldraw[white] (-0.1, -3.1) rectangle ++(3.2, 2.3);
				\filldraw[white] (-0.1, 12.1) rectangle ++(3.2, -2.3);
				
				\draw (0.5,10) edge [bend right = 140, very thick, looseness = 1.6, dashed] (0.5,-1);
				\draw (1.5,10) edge [bend right = 140, thick, looseness = 1.6, dashed] (1.5,-1);
				\draw (2.5,10) edge [bend right = 140, looseness = 1.6, dashed] (2.5,-1);
				
			\end{tikzpicture}	
			\captionsetup{justification=centering}
			\caption{Vertical cycles with \\ matchings between them}
		\end{subfigure}
		\caption{Two representations of the chain circuit $C_{4, 4}$.}
		\label{fig-ex-cc}
	\end{figure}
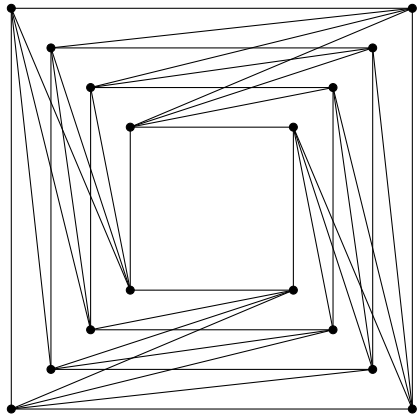
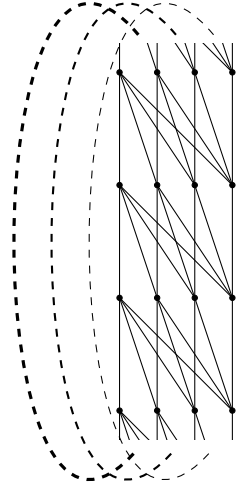

	The classes we are interested in are the class $\mathcal C_k$ of all graphs $C_{k, l}$ and their induced subgraphs, and the class $\widetilde{\mathcal C_k}$ of their CC-complements. We note that the minimal classes identified in both \cite{alecu-2p3} and \cite{primelet} are, up to some complementations, exactly the cycles-in-a-chain construction $C_{4, l}$. Our first result is that  $\mathcal C_k$ and $\widetilde{\mathcal C_k}$ have unbounded lettericity (that is, they are our classes of bounded $\lambda$ but unbounded lettericity). In view of \cite{alecu-2p3} and \cite{primelet}, the proof of this result can be seen as a standard method in the study of these problems. Indeed, the ideas behind the argument we present here extend even beyond the locally consistent setting, and will allow us to construct examples for the other two gaps as well.
	
	\begin{theorem}
		\label{thm-cc-unbdd-let}
		$\mathcal C_k$ and $\widetilde{\mathcal C_k}$ have unbounded lettericity.
	\end{theorem} 
	
	\begin{proof}
		Without loss of generality, we may restrict ourselves to letter graph expressions that use each letter in only one of the $k$ bags, since the minimum number of letters across all letter graph expressions is at most a factor of $k$ away from the minimum number of letters in expressions with this property. From our earlier discussion about complements, it suffices to prove the statement for $\mathcal C_k$. Suppose, for a contradiction, that the lettericity is bounded. If this is the case, then for any $t \in \mathbb N$, we can find $N \in \mathbb N$ such that $t$ cycles in some word representation of $C_{k, N}$ are represented by the same subword. In particular, those $t$ cycles form an induced $C_{k, t}$ using only $k$ letters. Since this can be done for any $t$, this shows the lettericity would in fact be bounded by at most $k$, and we should be able to express any $C_{k, i}$ with one letter per bag. Some quick case analysis shows the only way this would be possible is, up to symmetry, with a cyclic decoder $(a_i, a_{i + 1})$ (modulo $k$). However, it is impossible to represent even a single cycle $C_k$ in this way.
	\end{proof}
	
	\bigskip
	
	We are ready for the next step of our analysis. We will show that, in this restricted setting of chain circuits, the ``cycles-in-a-chain'' construction is the {\em only} obstacle to bounded lettericity.	Specifically, we show that in the universe of $k$-CCs, by forbidding any two chain circuits from $\mathcal C_k$ and $\widetilde{\mathcal C_k}$ respectively, we obtain a class of bounded lettericity.\footnote{Here, by ``forbidding chain circuits'' we mean that no copy of the underlying graphs appears as an induced subgraph {\em that respects the cyclic ordering of the bags}. We will skip the details.} Another way to state this is that $\mathcal C_k$ and $\widetilde{\mathcal C_k}$ are the only minimal classes of unbounded lettericity among $k$-CCs. Because of how those classes are defined, it suffices to show this in the case where we forbid $C_{k, i}$ and $\widetilde{C_{k, j}}$ for some $i, j \geq 1$. We will do the proof by induction on $i + j$. 
	
	In our proof, we will use a natural tool that we call the conflict graph:
	
	\begin{definition}
		Let $G$ be a $k$--chain circuit with CC-partition $\{A_i : 1 \leq i \leq k \}$. The \emph{conflict graph} $\conf(G)$ of $G$ is the directed graph with vertex set $V(G)$, and arcs $(v, w)$ whenever  
		\begin{itemize}
			\item $v \in A_i$, $w \in A_{i+1}$ for some $i$ (modulo $k$), and $\{v,w\} \in E(G)$, or
			\item $w \in A_i$, $v \in A_{i+1}$ for some $i$, and $\{v, w\} \notin E(G)$.
		\end{itemize} 
	\end{definition}
	
	The conflict graph of $G$ gives us a way of describing obstacles to representing $G$ with one letter per bag and a cyclic decoder. Specifically, conf$(G)$ has an arc from $v$ to $w$ exactly when the entry of $v$ is forced to appear before the entry of $w$ in such a representation. In particular, if such a representation of $G$ does exist, conf$(G)$ is acyclic.
	
	In fact, as one might expect, the converse is also true:
	
	\begin{lemma}
		\label{cycliciffdag}
		Suppose that $G$ is a $k$--chain circuit and that $\conf(G)$ is acyclic. Then there is a word $w$ on letters $a_1, \dots, a_k$ such that $G = G_{\mathcal D^c}(w)$, where $\mathcal D^c$ is the cyclic decoder $\{(a_i, a_{i+1}) : i \in \mathbb Z / k\mathbb Z\}$.
	\end{lemma} 
	
	\begin{proof}
		By standard results, conf$(G)$ admits a topological ordering $w$, i.e., a linear ordering of the vertices such that if $(x, y)$ is an arc, then $x$ comes before $y$ in $w$. It is routine to check that with a cyclic decoder, $w$ represents $G$.
	\end{proof}

	We are now ready to prove the result. We start with a base case for our induction.
	
	\begin{lemma} \label{basecase} 
		The subclass of $k$-chain circuits obtained by forbidding $C_{k, 1}$ and $\widetilde{C_{k, 1}}$ has bounded lettericity.
	\end{lemma}

	\begin{proof}
		From Lemma~\ref{cycliciffdag}, it is enough to show that if $G$ avoids the two graphs, then its conflict graph is acyclic. Suppose not, and find a shortest directed cycle $v_1, \dots, v_t = v_1$, with $v_r \in A_{i_r}$ (i.e., the $i_r$ are the indices of the bags successively visited by the cycle). Let us note a few facts about the sequence $i_r$.
		
		\begin{enumerate}[label = (\roman*)]
			
			\item $i_{r + 1} - i_r = \pm 1 \mod k$ for $1 \leq r < t$. 
			
			\item \label{wlogswitchdirection} We may assume without loss of generality that $i_1 = 1$. Indeed, we can permute the labels of the bags cyclically to make sure that this is the case. We may also assume that $i_2 = 2$, since otherwise we can work in $\widetilde{G}$ instead. Indeed, conf$(\widetilde{G})$ is just conf$(G)$ with all arcs reversed, and $G$ avoids $C_{k, 1}$ and $\widetilde{C_{k, 1}}$ if and only if its CC-complement does.
			
			\item For any $j \in \mathbb{Z}/k\mathbb{Z}$, $j$ and $j + 1$ appear consecutively at most once (and similarly for $j + 1$ and $j$). Indeed, suppose the cycle visits bags $j$ and $j + 1$ in that order twice, through vertices $v \in A_j, w \in A_{j + 1}$ the first time, and $v' \in A_j$, $w' \in A_{j + 1}$ the second time. Since $G[A_1 \cup A_2]$ induces a chain graph, and we know $v \sim w$ and $v' \sim w'$ in conf$(G)$, we must have $v \sim w'$ or $v' \sim w$. In either case, we have a ``shortcut'' through our cycle, which shows it is not minimal, contrary to our assumption. An analogous argument shows the statement for $j + 1$ and $j$.
			
			\item \label{horizontalindices} For any $j \in \mathbb{Z}/k\mathbb{Z}$, $j$ and $j + 1$ appear consecutively, or $j + 1$ and $j$ do. Indeed, suppose that there is a $j$ such that the cycle has no edge between $A_j$ and $A_{j + 1}$. We may assume, after changing our choice and doing some relabelling if necessary, that $j = k$, and that the cycle does pass through $A_1$. Let $v_1$ be the leftmost vertex of the cycle in $A_1$. Label its position in $A_1$ by 0. Further, label by 0 the position of the leftmost neighbour $v_2 \in A_2$ of $v_1$. Proceeding similarly, label by 0 the position of the leftmost neighbour $v_i \in A_i$ of $v_{i-1}$, for $i \leq k$. Note that if $v \in A_i$ ($1 \leq i < k$) has a non-negative label, then so do all its neighbours in $A_{i +1}$ by construction, and if $v \in A_i$ ($1 < i \leq k$) has a non-negative label, then all its non-neighbours in $A_{i - 1}$ have (strictly) positive labels. This means that in our set-up, the cycle cannot actually return to $v_1$.       
			
		\end{enumerate}
		The above observations imply that the sequence of $i_r$ is (up to starting the cycle at another point, working with $\widetilde{G}$ instead of $G$ and relabelling bags if necessary) either:
		
		\begin{itemize}
			\item $1, 2, \dots, k, 1$. In this case, $G$ contains a $C_{k, 1}$ (or a $\widetilde{C_{k, 1}}$ if we were working with $\widetilde{G}$, as described in \ref{wlogswitchdirection}).
			\item $1, 2, \dots, k, 1, k, \dots, 2, 1$, i.e., our cycle goes around the chain circuit, but instead of reaching its starting point $v_1$, it reaches another vertex $v' \in A_1$ before looping back around. However, this is impossible: $v'$ must be to the right of $v_1$ (otherwise the neighbour of $v'$ preceding it in the cycle is also adjacent to $v_1$ and we can find a shorter cycle), and we can use the same indexing argument as in \ref{horizontalindices} to conclude that the cycle cannot return to $v_1$. 
		\end{itemize}

	\end{proof}
	
	\begin{remark}\label{rem:top}
		One cannot help but notice a vague and superficial similarity between the arguments of the above proof and certain exercises from basic homotopy theory, such as determining the winding number of loops. This is perhaps a stretch of the imagination, but it would be interesting to look into whether there is something underlying this similarity. If the reader is willing to suspend their disbelief, let us dream together for a moment: what if there is an illuminating topological setting which gives elegant, satisfying interpretations and proofs for all the phenomena we are attempting to formalise and explain? It is, of course, possible, that questions related to this in a not necessarily obvious way were already asked and answered -- if the reader happens to have any insight or interest in this matter, they are kindly invited to contact the authors. 
	\end{remark}
	
	\begin{theorem} \label{thm-cc-main}
		Among $k$-chain circuits, the classes $\mathcal C_k$ and $\widetilde{\mathcal C_k}$ are the only minimal classes of unbounded lettericity.
	\end{theorem}
	
	\begin{proof}
		
		We have shown the base case in Lemma~\ref{basecase}. We need to show how the induction step works. Suppose thus that we have a $k$-chain circuit $G$ with chain partition $\{A_1, ..., A_k\}$, and with no $C_{k, p}$ and $\widetilde{C_{k, q}}$ for some $p, q \in \mathbb N$. We are going to split our chain circuit in a way which allows us to use the inductive hypothesis. We may assume without loss of generality that $G$ has a $C_{k, 1}$ as an induced subgraph (if not, consider its CC-complement; if its CC-complement also has no $C_{k, 1}$, then we are in the base case). Starting with any vertex $v_i \in A_i$ of the cycle, we colour the edge to its leftmost neighbour in $A_{i + 1}$ blue. We repeat this process with that leftmost neighbour, and keep doing this until we reach a vertex we have visited before. The process terminates, since any visited vertex in $A_i$ has a neighbour in $A_{i+1}$ (specifically, the appropriate vertex in the cycle we started with). We thus obtain a blue cycle $C_b$. Similarly, construct a red cycle $C_r$ by starting with $v_i$ and colouring in red the edge to its rightmost neighbour in $A_{i-1}$, and like before, repeating the process. See Figure~\ref{fig-red-blue-cycle} for an illustration.
		
		\begin{figure}[ht]
			\centering
			\begin{tikzpicture}[scale = 0.5, transform shape]
				

				\foreach \i in {-1,...,3} {
					\foreach \x in {-11,...,10} {
						\foreach \y in {\x, ..., 10} {
							\draw[lightgray] (\y, 3 * \i) -- (\x, 3 * \i + 3);
						}
					}
				}
				
				\filldraw[white] (-12.1, -3.1) rectangle ++(24.2, 2.3);
				\filldraw[white] (-12.1, 12.1) rectangle ++(24.2, -2.3);
				
				\begin{scope}
					\clip (-12.1, -3.1) rectangle ++(24.2, 2.3);
					\draw (-0.6,9.8) edge [bend right = 170, very thick, looseness = 2] (-0.6,-0.8);
					
					\draw (-9.1,9.8) edge [blue, bend right = 170, very thick, looseness = 2] (-9.1,-0.8);
					
					\draw (-2.9,9.8) edge [blue, out = 90, in = -110, very thick, looseness = 2] (-2.3,-0.8);
					
					\draw (-6.9,9.8) edge [blue, out = 90, in = -110, very thick, looseness = 2] (-6.3,-0.8);
					
					\draw (7.9,9.8) edge [red, bend right = 170, very thick, looseness = 2] (7.9,-0.8);
					
					\draw (1.1,9.8) edge [red, out = 90, in = -110, very thick, looseness = 2] (1.7,-0.8);
					
					\draw (5.1,9.8) edge [red, out = 90, in = -110, very thick, looseness = 2] (5.7,-0.8);
				\end{scope}
				
				\begin{scope}
					\clip (-12.1, 12.1) rectangle ++(24.2, -2.3);
					\draw (-0.6,9.8) edge [bend right = 170, very thick, looseness = 2] (-0.6,-0.8);
					
					\draw (-9.1,9.8) edge [blue, bend right = 170, very thick, looseness = 2] (-9.1,-0.8);
					
					\draw (-2.9,9.8) edge [blue, out = 90, in = -110, very thick, looseness = 2] (-2.3,-0.8);
					
					\draw (-6.9,9.8) edge [blue, out = 90, in = -110, very thick, looseness = 2] (-6.3,-0.8);
					
					\draw (7.9,9.8) edge [red, bend right = 170, very thick, looseness = 2] (7.9,-0.8);
					
					\draw (1.1,9.8) edge [red, out = 90, in = -110, very thick, looseness = 2] (1.7,-0.8);
					
					\draw (5.1,9.8) edge [red, out = 90, in = -110, very thick, looseness = 2] (5.7,-0.8);
				\end{scope}
				
				\begin{scope}
					\clip (-12.1, -0.8) rectangle ++(24.2, 10.4);
					\draw (-0.6,9.8) edge [bend right = 170, very thick, opacity = 0.3, looseness = 2, dashed] (-0.6,-0.8);
					
					\draw (-9.1,9.8) edge [blue, bend right = 170, very thick, opacity = 0.3, looseness = 2, dashed] (-9.1,-0.8);
					
					\draw (-2.9,9.8) edge [blue, out = 90, in = -110, very thick, opacity = 0.3, looseness = 2, dashed] (-2.3,-0.8);
					
					\draw (-6.9,9.8) edge [blue, out = 90, in = -110, very thick, opacity = 0.3, looseness = 2, dashed] (-6.3,-0.8);
					
					\draw (7.9,9.8) edge [red, bend right = 170, very thick, opacity = 0.3, looseness = 2, dashed] (7.9,-0.8);
					
					\draw (1.1,9.8) edge [red, out = 90, in = -110, very thick, opacity = 0.3, looseness = 2, dashed] (1.7,-0.8);
					
					\draw (5.1,9.8) edge [red, out = 90, in = -110, very thick, opacity = 0.3, looseness = 2, dashed] (5.7,-0.8);

				\end{scope}
				
				\draw[blue, very thick] (-0.5, 4.5) -- (-1, 3);
				\draw[blue, very thick] (-1, 3) -- (-2, 0);
				\draw[blue, very thick] (-3, 9) -- (-4, 6);
				\draw[blue, very thick] (-4, 6) -- (-4.25, 5.25);
				\draw[blue, very thick, loosely dotted] (-4.25, 5.25) -- (-4.75, 3.75);
				\draw[blue, very thick] (-4.75, 3.75) -- (-5, 3);
				\draw[blue, very thick] (-5, 3) -- (-6, 0);
				\draw[blue, very thick] (-7, 9) -- (-8, 6);
				\draw[blue, very thick] (-8, 6) -- (-9, 3);
				
				\draw[blue, very thick] (-9, 9) -- (-9, 0);
				\draw (-9, 9) edge[blue, very thick, out = 90, in = -80] (-9.1, 9.8);
				\draw (-9, 0) edge[blue, very thick, out = -90, in = 80] (-9.1, -0.8);
				
				\draw (-2, 0) edge[blue, very thick, in = 70, looseness = 0.1] (-2.3, -0.8);
				\draw (-3, 9) edge[blue, very thick, out = 80, in = 90, looseness = 0.1] (-2.9, 9.8);
				
				\draw (-6, 0) edge[blue, very thick, in = 70, looseness = 0.1] (-6.3, -0.8);
				\draw (-7, 9) edge[blue, very thick, out = 80, in = 90, looseness = 0.1] (-6.9, 9.8);
				
				\draw[red, very thick] (-0.5, 4.5) -- (0, 6);
				\draw[red, very thick] (0, 6) -- (1, 9);
				\draw[red, very thick] (2, 0) -- (3, 3);
				\draw[red, very thick] (3, 3) -- (3.25, 3.75);
				\draw[red, very thick, loosely dotted] (3.25, 3.75) -- (3.75, 5.25);
				\draw[red, very thick] (3.75, 5.25) -- (4, 6);
				\draw[red, very thick] (4, 6) -- (5, 9);
				\draw[red, very thick] (6, 0) -- (7, 3);
				\draw[red, very thick] (7, 3) -- (8, 6);
				
				\draw[red, very thick] (8, 9) -- (8, 0);
				\draw (8, 9) edge[red, very thick, out = 90, in = -80] (7.9, 9.8);
				\draw (8, 0) edge[red, very thick, out = -90, in = 80] (7.9, -0.8);
				
				\draw (2, 0) edge[red, very thick, in = 70, looseness = 0.1] (1.7, -0.8);
				\draw (1, 9) edge[red, very thick, out = 80, in = 90, looseness = 0.1] (1.1, 9.8);
				
				\draw (6, 0) edge[red, very thick, in = 70, looseness = 0.1] (5.7, -0.8);
				\draw (5, 9) edge[red, very thick, out = 80, in = 90, looseness = 0.1] (5.1, 9.8);
				
				\draw[very thick] (-0.5, 9) -- (-0.5, 0);
				\draw (-0.5, 9) edge[very thick, out = 90, in = -80] (-0.6, 9.8);
				\draw (-0.5, 0) edge[very thick, out = -90, in = 80] (-0.6, -0.8);
				
			\end{tikzpicture}

			\caption{The red and blue spirals}
			\label{fig-red-blue-cycle}
		\end{figure}
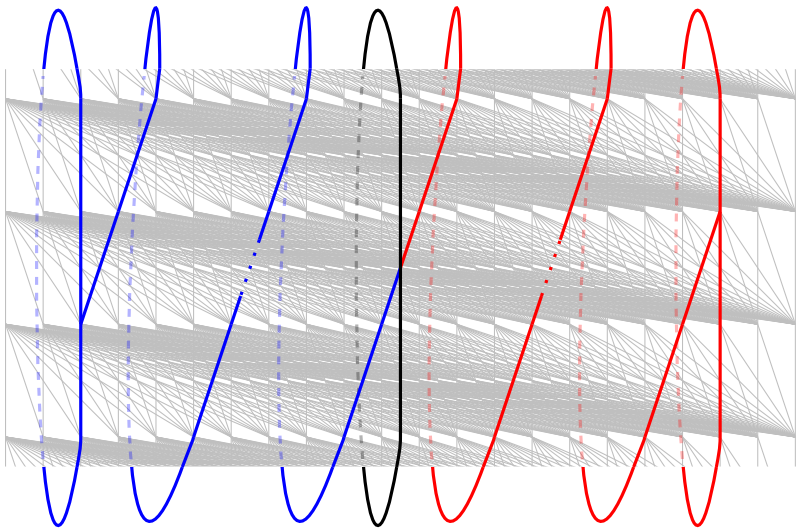
		
		Now let $G^L$ be the induced subgraph strictly to the left of the blue cycle, let $G^R$ be the subgraph strictly to the right of the red cycle, and let $G^M$ be the middle subgraph (including the red and blue cycles). That is, for each $A_i$, $G^L$ contains the vertices of $A_i$ strictly to the left of the vertex in $C_b \cap A_i$, $G^R$ contains the vertices strictly to the right of the vertex in $C_r \cap A_i$, and $G^M$ contains the remaining vertices. Write $G^L_i := G^L \cap A_i$ (extending the notation to $G^R$ and $G^M$ in the obvious way).
		
		Now notice that by construction, each of $G^L$ and $G^R$ is $C_{k, p - 1}$-free, since a $C_{k, p - 1}$ in $G^L$ together with the blue cycle would give a $C_{k, p}$, and similarly for a $C_{k, p - 1}$ in $G^R$ together with the red cycle. Hence the inductive hypothesis applies, and there is a $c$ (depending on $p$) such that $G^L$ and $G^R$ can each be represented by a word with $c$ letters. 
		
		Moreover, the edges between the three parts are easy to describe: for any $i$, we have no edges between $G^R_i$ and $G^M_{i + 1} \cup G^L_{i + 1}$, because of how we constructed our red cycle, and no edges between $G^L_i$ and $G^M_{i - 1} \cup G^R_{i - 1}$ because of how we constructed the blue cycle. We also have all possible edges between $G^L_i$ and $G^M_{i + 1} \cup G^R_{i+1}$, as well as all possible edges between $G^R_i$ and $G^M_{i - 1} \cup G^L_{i - 1}$ because of the properties of chain circuits. Given these structural features, it is clear that if we have words representing each of $G^L$, $G^M$ and $G^R$ with $c_1, c_2$ and $c_3$ letters respectively, we can construct a word representing $G$ using $c_1 + c_2 + c_3$ letters with a carefully chosen decoder.   
		
		To prove the theorem, it remains to show that we can express $G^M$ using a number of letters that only depends on $p$. Before we do this, note that, although the blue and red cycles are not uniquely defined (we might get different red and blue cycles if we choose a different starting cycle), the partition we get at the end into $G^L, G^M$ and $G^R$ still satisfies the properties we have described so far. In particular, we may assume without loss of generality that the cycle we start with is given by vertices $v_1, \dots, v_k$ with $v_i \in A_i$, such that $v_{i+1}$ is the leftmost neighbour of $v_i$ for $i$ modulo $k$ (in other words, it {\em is} the blue cycle). 
		
		Now start with $v_k \in C_b$ and consider the sequence of vertices that we get by repeatedly taking the rightmost vertex in the previous bag modulo $k$. Label the vertices of this sequence as $v_{i, j}$, where $i$ is the bag of the vertex, and $j$ counts the number of times the sequence has visited that bag after this vertex (so $v_k$ becomes $v_{k, 1}$, and the sequence continues with $v_{k-1, 1}, \dots, v_{1, 1}, v_{k, 2}, \dots$). One can picture this sequence as a spiral winding around the chain circuit, ending at the red cycle. Note that this spiral is indeed  winding ``to the right'', more precisely, for all $i$, if $j_2 > j_1$, $v_{i, j_2}$ is to the right of $v_{i, j_1}$ in the usual CC-ordering (or the two vertices are equal). To see this, observe that if a vertex of the sequence is to the right of a cycle (i.e., to the right of the vertex of the cycle lying in the corresponding bag), so are all following vertices.
		
		Put $S_0 := C_b$, and let $S_j$ ($j > 0$) be the cycles induced by vertices $v_{i, j}$ (note that those are indeed cycles, since the chain circuit property together with the above observation implies that $v_{k, j}$ is adjacent to $v_{1, j}$). Write $G^M_j$ for the set of vertices strictly between $S_j$ and $S_{j + 1}$. Finally, write $A_{i, j}$ for $A_i \cap G^M_j$. 	
		
		The following statements hold:
		
		\begin{itemize}
			
			\item $V(G^M) = \bigcup\limits_{\substack{1 \leq i \leq k \\ j \geq 0}} A_{i, j} \cup \bigcup\limits_{j \geq 0} S_j$.
			
			\item The number of disjoint $S_j$ is bounded as a function of $p$. To see why, assume that $S_1, S_2, \dots, S_{2r + 1}$ are all disjoint for some $r \in \mathbb N$. One can check that, by construction, $S_1, S_3, \dots, S_{2r - 1}, S_{2r + 1}$ induce a $C_{k, r}$. As $C_{k, p}$ is forbidden, it follows that $r < p$. Since the sequence $(v_{i, j})$ becomes periodic once it repeats a vertex, it follows that $$|\{v_{i, j} : i \geq 1, j \geq 0\}| < k(2p + 1).$$  
			
		\end{itemize}
		
		From \cite{letter-graphs}, we know that if the lettericity of a graph is $l$, then adding a vertex produces a graph of lettericity at most $2l + 1$. In view of the second statement above, $\bigcup\limits_{j \geq 0} S_j$ has size at most $k(2p + 2)$, so it suffices to show that $G^M \setminus \bigcup\limits_{j \geq 0} S_j$ can be expressed using a bounded number of letters. We claim that this can be done using one letter $a_{i, j}$ for each non-empty set $A_{i, j}$. This is enough, since from the above discussion, $|\{ (i, j) : A_{i, j}\neq \emptyset \}| < k(2p + 2)$. To see how one can construct a word and decoder representing $G^M \setminus \bigcup\limits_{j \geq 0} S_j$, let us  examine the edges between those sets. Let $A_{i_1, j_1}$ and $A_{i_2, j_2}$ be two such sets. If $i_1 - i_2 \neq \pm 1 \mod k$, then there are no edges between $A_{i_1, j_1}$ and $A_{i_2, j_2}$, so write $i$ for $i_1$ and assume without loss of generality that $i_2 = i + 1 \mod k$. All of the following claims follow straightforwardly from our construction of the $A_{i, j}$ and the properties of chain circuits.
		
		\begin{itemize}
			\item If $i \neq k$, we have:
			\begin{itemize}
				\item No edges between $A_{i, j_1}$ and $A_{i + 1, j_2}$ when $j_2 < j_1$.
				\item All possible edges $A_{i, j_1}$ and $A_{i + 1, j_2}$ when $j_2 > j_1$
			\end{itemize} 
			
			\item For $i = k$, we have:
			
			\begin{itemize}
				\item No edges between $A_{k, j_1}$ and $A_{1, j_2}$ when $j_2 < j_1 - 1$
				\item All possible edges between $A_{k, j_1}$ and $A_{1, j_2}$ when $j_2 \geq j_1$. 
			\end{itemize}
			
		\end{itemize}
		
		The key is now to notice that all of the non-trivial adjacencies appear between consecutive bags in the following sequence that ``spirals around'' the chain circuit (some of the bags might be empty):
		$$A_{k, 0}, A_{k - 1, 0}, \dots, A_{1, 0}, A_{k, 1}, A_{k-1, 1}, \dots, A_{1, 1}, A_{k, 2}, \dots.$$
		
		One helpful way of conceptualising this is by thinking of the red edges as ``impermeable'' to edges crossing them from top right to bottom left, and to non-edges crossing them from top left to bottom right.
		
		Any two consecutive bags in the above sequence induce chain graphs. As described in \cite{3let}, we can realise the subgraph consisting of the edges between consecutive bags by using one letter $a_{i, j}$ for each bag $A_{i, j}$; the decoder then contains the (ordered) pairs of letters corresponding to consecutive bags. The construction of the word can be done inductively: if we have a word describing the subgraph up to a certain bag $A_{i, j}$ in the sequence, we can add the letters corresponding to vertices from the next bag by placing them carefully among the letters of bag $A_{i, j}$ (see \cite{3let} for more details). To make the word represent $G^M \setminus \bigcup\limits_{j \geq 0} S_j$, all we need to do is add pairs $(a_{i_1, j_1}, a_{i_2, j_2})$ and $(a_{i_2, j_2}, a_{i_1, j_1})$ to the decoder whenever the corresponding bags have all possible edges between them.
		
	\end{proof}
	
	Theorem~\ref{thm-cc-main} gives us an answer to Problems~\ref{prob:bddlet}~and~\ref{lambda-let} in a restricted setting. Moreover, it does so in a fairly constructive way: suppose we were working within the monotone grid class of a PMM with cyclic cell graph. Then, given a gridded permutation $\pi$, we could, in principle, carefully follow the above proof and add a bounded number of lines to the gridding to make it geometric. We believe that, in fact, this ``chains-in-a-cycle'' construction and its complements are the only obstacles to bounded lettericity:
	
	\begin{conjecture}
		If $\mathcal X$ is a class of bounded $\lambda$ and unbounded lettericity, then $\mathcal X$ contains the class $\mathcal C_k$ for some $k$ (or one of its complements).
	\end{conjecture}  

	So where exactly is the difficulty in generalising the result? Returning to our permutation setting, the problem is that, as soon as the cell graph contains more than one cycle, things go awry.\footnote{On par with Remark~\ref{rem:top}, we note that the difficulty of the problem depends on the topology of the cell graph.} To demonstrate, let us consider a specific setting. The easiest unsolved instance of the problem occurs in the monotone grid class of the matrix $\begin{pmatrix} 1 & 1 & 0 \\ 1 & 1 & 1 \\ 0 & 1 & 1\end{pmatrix}.$ Identifying all obstacles to to geometric griddability in this class will be a big step forward towards our goal.\footnote{The authors are optimistic that it is only a matter of time before progress is made in this direction; either we will manage to show that forbidding the cycles-in-a-chain obstacles is enough to guarantee geometric griddability in this class, or we will find another type of obstacle.} There are two ``simple cycles'' in the cell graph of the matrix: the one given by the four entries in the top left, and the one given by the four entries in the bottom right. Call them $A$ and $B$ respectively. We may use (the constructive version of) Theorem~\ref{thm-cc-main} as a black box: if we forbid $C_{4, p}$ and complements ($p \in \mathbb N$), we are able to ``clean out'' all circuits that lie entirely in one of $A$ or $B$. Any other obstacle to geometric griddability must somehow involve both $A$ and $B$.
	
	So what kinds of obstacles are there that can involve both $A$ and $B$? It turns out that our cycles-in-a-chain make another appearance here. Imagine working in a ``virtual'' chain circuit that starts in the middle cell of the matrix, goes around $A$ once, then around $B$; the middle cell thus occurs twice as a bag, but we treat those occurrences as disjoint copies. One may then construct new cycles-in-a-chain structures such as the one depicted in Figure~\ref{fig-obstacle-8}: we go around cycle $A$ (represented by red arcs), arrive to the left of where we started, then go around cycle $B$ (represented by blue arcs) to get back to the start. In order to eliminate those structures, we may use Theorem~\ref{thm-cc-main} applied to this virtual chain circuit, provided $C_{8, p}$ is forbidden (with complements).

	
	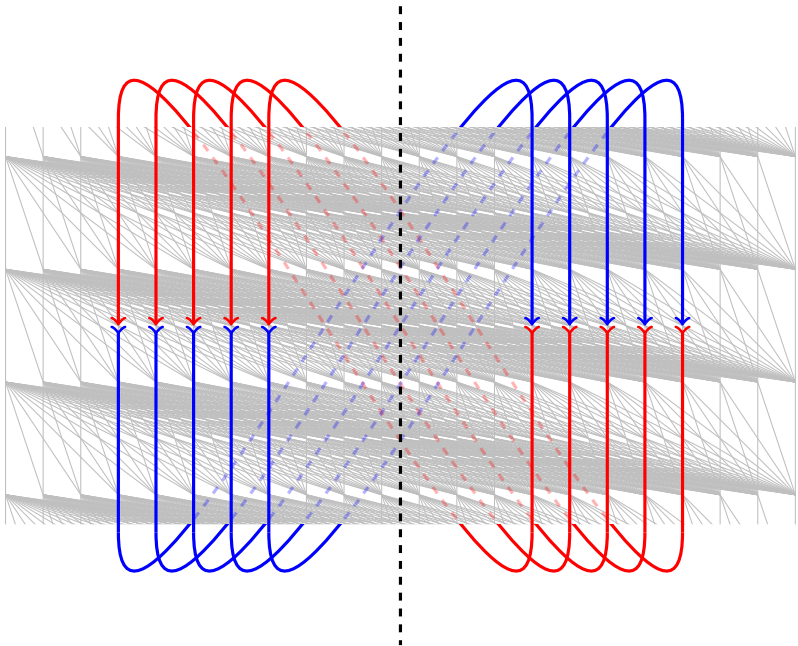
\begin{figure}[ht]
		\centering
		\begin{tikzpicture}[scale = 0.5, transform shape]
			

			\foreach \i in {-1,...,3} {
				\foreach \x in {-11,...,10} {
					\foreach \y in {\x, ..., 10} {
						\draw[lightgray] (\y, 3 * \i) -- (\x, 3 * \i + 3);
					}
				}
			}
			
			\filldraw[white] (-12.1, -3.1) rectangle ++(24.2, 2.3);
			\filldraw[white] (-12.1, 12.1) rectangle ++(24.2, -2.3);
			
			\begin{scope}
				\clip (-12.1, -3.1) rectangle ++(24.2, 2.3);
				
				\draw (3, -1) edge[red, very thick, out = -90, in = 90] (-8, 10);
				\draw (4, -1) edge[red, very thick, out = -90, in = 90] (-7, 10);
				\draw (5, -1) edge[red, very thick, out = -90, in = 90] (-6, 10);
				\draw (6, -1) edge[red, very thick, out = -90, in = 90] (-5, 10);
				\draw (7, -1) edge[red, very thick, out = -90, in = 90] (-4, 10);
				
				\draw (-8, -1) edge[blue, very thick, out = -90, in = 90] (3, 10);
				\draw (-7, -1) edge[blue, very thick, out = -90, in = 90] (4, 10);
				\draw (-6, -1) edge[blue, very thick, out = -90, in = 90] (5, 10);
				\draw (-5, -1) edge[blue, very thick, out = -90, in = 90] (6, 10);
				\draw (-4, -1) edge[blue, very thick, out = -90, in = 90] (7, 10);
				
			\end{scope}
			
			\begin{scope}
				\clip (-12.1, 12.1) rectangle ++(24.2, -2.3);
				
				\draw (3, -1) edge[red, very thick, out = -90, in = 90] (-8, 10);
				\draw (4, -1) edge[red, very thick, out = -90, in = 90] (-7, 10);
				\draw (5, -1) edge[red, very thick, out = -90, in = 90] (-6, 10);
				\draw (6, -1) edge[red, very thick, out = -90, in = 90] (-5, 10);
				\draw (7, -1) edge[red, very thick, out = -90, in = 90] (-4, 10);
				
				\draw (-8, -1) edge[blue, very thick, out = -90, in = 90] (3, 10);
				\draw (-7, -1) edge[blue, very thick, out = -90, in = 90] (4, 10);
				\draw (-6, -1) edge[blue, very thick, out = -90, in = 90] (5, 10);
				\draw (-5, -1) edge[blue, very thick, out = -90, in = 90] (6, 10);
				\draw (-4, -1) edge[blue, very thick, out = -90, in = 90] (7, 10);
				
			\end{scope}
			
			\begin{scope}
				\clip (-12.1, -0.8) rectangle ++(24.2, 10.4);
				
				\draw (3, -1) edge[red, very thick, out = -90, in = 90, opacity = 0.3, dashed] (-8, 10);
				\draw (4, -1) edge[red, very thick, out = -90, in = 90, opacity = 0.3, dashed] (-7, 10);
				\draw (5, -1) edge[red, very thick, out = -90, in = 90, opacity = 0.3, dashed] (-6, 10);
				\draw (6, -1) edge[red, very thick, out = -90, in = 90, opacity = 0.3, dashed] (-5, 10);
				\draw (7, -1) edge[red, very thick, out = -90, in = 90, opacity = 0.3, dashed] (-4, 10);
				
				\draw (-8, -1) edge[blue, very thick, out = -90, in = 90, opacity = 0.3, dashed] (3, 10);
				\draw (-7, -1) edge[blue, very thick, out = -90, in = 90, opacity = 0.3, dashed] (4, 10);
				\draw (-6, -1) edge[blue, very thick, out = -90, in = 90, opacity = 0.3, dashed] (5, 10);
				\draw (-5, -1) edge[blue, very thick, out = -90, in = 90, opacity = 0.3, dashed] (6, 10);
				\draw (-4, -1) edge[blue, very thick, out = -90, in = 90, opacity = 0.3, dashed] (7, 10);
				
			\end{scope}
			
			
			\draw[>-, blue, very thick] (-4, 4.5) -- (-4, -1);
			\draw[>-, blue, very thick] (-5, 4.5) -- (-5, -1); 
			\draw[>-, blue, very thick] (-6, 4.5) -- (-6, -1); 
			\draw[>-, blue, very thick] (-7, 4.5) -- (-7, -1); 
			\draw[>-, blue, very thick] (-8, 4.5) -- (-8, -1); 
			
			\draw[<-, red, very thick] (-4, 4.5) -- (-4, 10);
			\draw[<-, red, very thick] (-5, 4.5) -- (-5, 10);
			\draw[<-, red, very thick] (-6, 4.5) -- (-6, 10);
			\draw[<-, red, very thick] (-7, 4.5) -- (-7, 10);
			\draw[<-, red, very thick] (-8, 4.5) -- (-8, 10);
			
			\draw[<-, blue, very thick] (3, 4.5) -- (3, 10);
			\draw[<-, blue, very thick] (4, 4.5) -- (4, 10);
			\draw[<-, blue, very thick] (5, 4.5) -- (5, 10);
			\draw[<-, blue, very thick] (6, 4.5) -- (6, 10);
			\draw[<-, blue, very thick] (7, 4.5) -- (7, 10);
			
			\draw[>-, red, very thick] (3, 4.5) -- (3, -1);
			\draw[>-, red, very thick] (4, 4.5) -- (4, -1);
			\draw[>-, red, very thick] (5, 4.5) -- (5, -1);
			\draw[>-, red, very thick] (6, 4.5) -- (6, -1);
			\draw[>-, red, very thick] (7, 4.5) -- (7, -1);

			\draw[very thick, dashed] (-0.5, 13) -- (-0.5, -4);
			
		\end{tikzpicture}

		\caption{A ``figure 8'' obstacle}
		\label{fig-obstacle-8}
	\end{figure}
	
	
	We note that this kind of obstacle is still, in some sense, periodic. What makes the problem messier at this stage is the possibility that aperiodic obstacles exist. One can think of them by analogy with the aperiodic fundamental antichains described in \cite[Section~5]{murphy-profile}. The general idea of the construction presented there is that we produce an aperiodic word on $A$ and $B$, and go around the cycles in the order indicated by the word.\footnote{The construction in \cite{murphy-profile} is slightly more involved than that, but the details do not matter for this discussion.} Since geometrically griddable classes are wqo, we had better be able to destroy all but a finite number of the elements in such an antichain by forbidding the appropriate graphs. Could it be that those graphs are just our old cycles-in-a-chain, or is some fundamentally different kind of obstacle hiding somewhere in there?
	
	\medskip
	
	Let us try to start answering that question. Sticking with the above example, consider a walk in the conflict graph.\footnote{We note that it makes sense to define the conflict graph in any locally consistent setting.} When we start in a bag and go around one of the simple cycles in the cell graph of the matrix, we end up somewhere in the original bag. By applying (the constructive version of) Theorem~\ref{thm-cc-main} on the (boundedly many) simple cycles, we gain some control of where we end up: for instance, by construction, we may not return to where we started by just walking repeatedly around the same simple cycle in the original cell graph. However, as in Figure~\ref{fig-obstacle-8}, we might be able to return to the start (and thus produce a circuit in the conflict graph) by walking around different cycles of the cell graph in succession. This means that, in order to eliminate all circuits, we need to apply Theorem~\ref{thm-cc-main} to the ``virtual chain circuits'' coming from certain concatenations of simple cycles in the line graph.\footnote{Additionally, we need to distinguish between the two possible directions in the simple cycles.} If we manage to show that it suffices to do it only for a bounded number of such concatenations, we are done. 
	
	One way to go about this would be by stating and proving some kind of Ramsey-type result: if our (not necessarily chordless) circuit in the original conflict graph is long enough, the sequence of cells visited is going to have some repeated subsequences. We would want to use this in order to find a certain structure that would have been destroyed by our bounded number of applications of Theorem~\ref{thm-cc-main}.\footnote{In fact, this is how we originally found the example from Figure~17: a certain number of ``backwards paths'' like the red ones from the figure guarantees by simple arguments that certain configurations of their endpoints are unavoidable; one such configuration leads to the given obstacle.} Indeed, it would be enough to show that any circuit above a certain size in the original conflict graph had an arc that was eliminated by one of a bounded number of applications of Theorem~\ref{thm-cc-main}.
	
	\medskip
	
	Let us also talk a bit about scalability. Suppose we have completely solved the problem in the above instance: how do we generalise the solution? We do not yet know, but we have a perspective that might prove helpful in that regard. It begins from the noteworthy remark that there is a finite set of ``obvious candidates'' for a global order -- in the world of permutations, they are the ones coming from spanning trees of the cell graph. Indeed, if the cell graph is already a tree, the unique candidate actually happens to work (Theorem~\ref{thm-forest-wqo}). In the $k$-cyclic setting, it is possible to reinterpret the problem as follows: fix a spanning path of the cell graph; from this, using the methods from \cite{vatter-pwo-monotone}, we may produce in a systematic way a total order on the elements of the permutation. All arcs of the conflict graph between consecutive cells along the path are ``forward'', in the sense that they agree with the order. The only ``backward'' arcs may only occur between the cells corresponding to the ends of the path (i.e., along the unique edge of the cell graph not belonging to the spanning tree). It should be possible to rewrite Theorem~\ref{thm-cc-main} in this setting, where the focus is on controlling those backward arcs. In the general case, we would be looking at the order given by a spanning tree,\footnote{Could the choice of the spanning tree matter? Perhaps we would have something to gain by considering multiple spanning trees simultaneously.} and the backward arcs would appear along the edges not in it. While the fundamental difficulties we encountered before are still there, it is possible that a spanning tree/backward arc-focused perspective gives a cleaner interpretation of the problem, and thus a better path towards a solution.
	
	\medskip
	
	One final item that we would like to mention is the fact that, while the ``figure 8'' obstacle described above looks identical to the cyclic ones from a graph perspective, what happens in the conflict graph is slightly different. Indeed, the cyclic obstacles from Theorem~\ref{thm-cc-main} induce cycles-in-a-chain in the conflict graph; however, in the ``figure 8'' example, those cycles-in-a-chain are {\em not} induced in the conflict graph -- while the two occurrences of the middle cell are independent in the virtual chain circuit, in the actual conflict graph, the two sets are complete to each other, with all arcs oriented in the same direction. This raises the question of whether we are stating our problems and conjectures in the most natural way. We will return to this in Section~\ref{sec:loh}; but before that, let us talk a bit about the other two transitions at play. Rather than going through them ``in order'', we next examine the one between $\gamma$ and $\sigma$.

	\subsection{Between bounded $\gamma$ and bounded $\sigma$}
	
%
%
	
	As before, we start by presenting a class which has bounded $\gamma$ but unbounded $\sigma$. Let us try to construct an example of this type as simple as possible. Consider a graph $G$ whose vertex set consists of three independent sets $A$, $B$ and $C$ on $n$ vertices each, and assume $G[A \cup B]$ and $G[B \cup C]$ are prime chain graphs, while $G[A \cup C]$ is edgeless. In the scope of the current discussion, we will call graphs with this structure {\em linked chain graphs}. Let us order the vertices in $B$ in increasing order with respect to their neighbourhoods in $A$, and label them by $1, \dots, n$. Since $G[B \cup C]$ is a prime chain graph, there is a unique permutation $\pi \in S_n$ such that the ordering $\pi(1), \dots, \pi(n)$ has decreasing neighbourhoods in $C$. We call $\pi$ the {\em linking permutation} of $G$. See Figure~\ref{fig-mixed-up} for an illustration.
	
	\medskip
	
	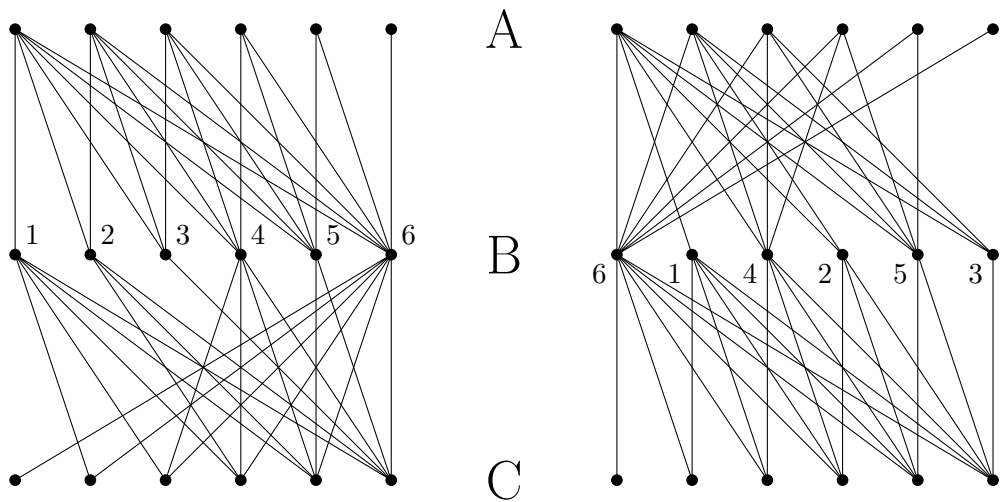
\begin{figure}[ht]
		\centering
		
		\begin{tikzpicture}[scale = 1, transform shape]
			
			\foreach \i in {0,...,5} {
				\foreach \x in {0,...,2} {
					\filldraw (\i, 3 * \x) circle (2pt);
				}
			}
			
			\foreach \i in {0,...,5} {
				\foreach \x in {\i,..., 5} {
					\draw (\i, 6) -- (\x, 3);
				}
			}
			
			\draw 
			(5, 3) -- (0, 0)
			(5, 3) -- (1, 0)
			(5, 3) -- (2, 0)
			(5, 3) -- (3, 0)
			(5, 3) -- (4, 0)
			(5, 3) -- (5, 0);
			
			\draw 
			(0, 3) -- (1, 0)
			(0, 3) -- (2, 0)
			(0, 3) -- (3, 0)
			(0, 3) -- (4, 0)
			(0, 3) -- (5, 0);
			
			\draw
			(3, 3) -- (2, 0)
			(3, 3) -- (3, 0)
			(3, 3) -- (4, 0)
			(3, 3) -- (5, 0);
			
			\draw
			(1, 3) -- (3, 0)
			(1, 3) -- (4, 0)
			(1, 3) -- (5, 0);
			
			\draw
			(4, 3) -- (4, 0)
			(4, 3) -- (5, 0);
			
			\draw
			(2, 3) -- (5, 0);
			
			\foreach \i in {8,...,13} {
				\foreach \x in {0,...,2} {
					\filldraw (\i, 3 * \x) circle (2pt);
				}
			}
			
			\foreach \i in {8,...,13} {
				\foreach \x in {\i,..., 13} {
					\draw (\i, 3) -- (\x, 0);
				}
			}
			
			\draw
			(8, 3) -- (8, 6)
			(8, 3) -- (9, 6)
			(8, 3) -- (10, 6)
			(8, 3) -- (11, 6)
			(8, 3) -- (12, 6)
			(8, 3) -- (13, 6);
			
			\draw
			(12, 3) -- (8, 6)
			(12, 3) -- (9, 6)
			(12, 3) -- (10, 6)
			(12, 3) -- (11, 6)
			(12, 3) -- (12, 6);
			
			\draw
			(10, 3) -- (8, 6)
			(10, 3) -- (9, 6)
			(10, 3) -- (10, 6)
			(10, 3) -- (11, 6);
			
			\draw
			(13, 3) -- (8, 6)
			(13, 3) -- (9, 6)
			(13, 3) -- (10, 6);
			
			\draw
			(11, 3) -- (8, 6)
			(11, 3) -- (9, 6);
			
			\draw
			(9, 3) -- (8, 6);

			\draw (6.5, 6) node {\huge{A}};
			\draw (6.5, 3) node {\huge{B}};
			\draw (6.5, 0) node {\huge{C}};
			
			\draw 
			(0, 3) node[above right] {1}
			(1, 3) node[above right] {2}
			(2, 3) node[above right] {3}
			(3, 3) node[above right] {4}
			(4, 3) node[above right] {5}
			(5, 3) node[above right] {6};
			
			\draw 
			(8, 3) node[below left] {6}
			(9, 3) node[below left] {1}
			(10, 3) node[below left] {4}
			(11, 3) node[below left] {2}
			(12, 3) node[below left] {5}
			(13, 3) node[below left] {3};
			
		\end{tikzpicture}

		\caption{A graph with linking permutation 614253}
		\label{fig-mixed-up}
	\end{figure}
	
	It is clear that every linked chain graph $G$ has $\gamma(G) \leq 3$, and that $G$ is uniquely determined by its linking permutation. Moreover, if $\pi$ is a subpattern of $\rho$, then the linked chain graph corresponding to $\rho$ contains an induced copy of the linked chain graph corresponding to $\pi$.\footnote{An easy way to see this is by noting that, in the drawing on the left-hand side of Figure~\ref{fig-mixed-up}, the ``leftmost edges'' in the bottom layer are a line segment intersection model for $\pi$.} Can we construct a sequence of permutations $\pi_n$ such that the corresponding sequence of linked chain graphs $G_n$ has unbounded $\sigma$? Note that if the sequence of permutation graphs $G_{\pi_n}$ has bounded chromatic number (that is, if each $\pi$ can be partitioned into a fixed number of increasing subsequences), say at most $t$, then $\sigma(G_n) \leq t + 2$. Indeed, we may partition $B$ into $t$ smaller bags, each corresponding to one of the increasing subsequences, and one easily checks that this new partition is locally semi-consistent.
	
	A similar argument applies if $G_{\pi_n}$ has bounded co-chromatic number. Let us thus look for permutations whose family of permutation graphs has unbounded co-chromatic number. Such families are easy to construct. For instance, we may let $\pi_n$ be the permutation on $n^2$ elements given by the concatenation $w_1w_2\dots w_n$, where $w_i$ lists the elements of $\{x : 1 \leq x \leq n^2\text{ and }x \equiv i \mod n\}$ in decreasing order (see Figure~\ref{fig-exampleset}). For the sequence $(\pi_n)_{n \geq 1}$, one checks that the size of the maximum homogeneous set in the corresponding permutation graphs is sublinear, which immediately implies unbounded co-chromatic number. 
	
	\begin{figure}[ht]
		\begin{center}
			\begin{tikzpicture}[scale=0.6, transform shape]
				
				\foreach \i in {0,...,6}
				{
					\foreach \x in {0,...,6}
					{
						\filldraw (\i - \x/7, \i/7 + \x) circle (2pt) node{}; 
					}
				}
				\draw[dotted] (0, 0) -- (0, 7);
				\draw[dotted] (0, 0) -- (7, 0);
				
			\end{tikzpicture}
		\end{center}
		\caption{The permutation $\pi_7$}
		\label{fig-exampleset}
	\end{figure}
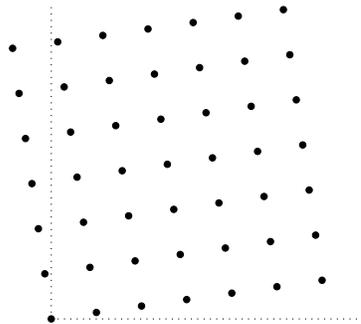

	To show that the corresponding family of linked chain graphs has indeed unbounded $\sigma$, we first need a corollary to van der Waerden's theorem on arithmetic progressions. Let us start by recalling the theorem: 
	
	\begin{theorem}[van der Waerden's Theorem \cite{van-der-waerden-beweis}]
		For any $p, k \in \mathbb N$, there exists a number $N \in \mathbb N$ such that if $[N]$ is coloured with $p$ different colours, then there are at least $k$ integers in arithmetic progression whose elements are the same colour.
	\end{theorem}
	
	\begin{corollary}\label{cor-vdw}
		For any $p, k \in \mathbb N$, there exists a number $N \in \mathbb N$ such that if $[N] \times [N]$ is coloured with $p$ different colours, then there are two arithmetic progressions $X = \{x_1, \dots, x_s\}$ and $Y = \{y_1, \dots, y_t\}$ of length at least $k$ such that $X \times Y$ is monochromatic.
	\end{corollary}
	
	\begin{proof}
		We first use van der Waerden's theorem to find a number $N_1$ such that any colouring of $[N_1]$ with $p$ colours contains a monochromatic arithmetic progression of length at least $k$. Note that there are at most ${N_1 \choose 2}$ possibilities for that arithmetic progression (it is uniquely determined by its first two terms). We then use van der Waerden's theorem a second time to find a number $N_2$ such that any colouring of $[N_2]$ with $p {N_1 \choose 2}$ colours contains an arithmetic progression of length at least $k$. 
		
		Suppose now that $[N_1] \times [N_2]$ is coloured with $p$ colours. By choice of $N_1$, for each $1 \leq i \leq N_2$, the set $[N_1] \times i$ contains a monochromatic arithmetic progression $X'$ of length at least $k$. Colour $i$ with the pair $(X', c)$, where $c$ is the colour of $X'$. This gives a colouring of $[N_2]$ with (at most) $p {N_1 \choose 2}$ colours, and hence there is a monochromatic arithmetic progression $Y$ of length at least $k$. Suppose its colour is $(X, c)$; then by construction, $X \times Y$ is monochromatic with colour $c$. 
	\end{proof}
	
	We apply Corollary~\ref{cor-vdw} to the plots of the permutations $\pi_n$. Indeed, the corollary directly implies the following: for any $p$ and $t$, there exists $N$ such that whenever we partition the plot of $\pi_N$ into at most $p$ pieces, one of the pieces contains $\pi_t$ as a subpattern. With this, we are ready to prove our result.
	
	\begin{theorem} \label{thm-gamma-lambda}
		Let $\pi_n$ be the permutation on $n^2$ elements given by the concatenation $w_1w_2\dots w_n$, where $w_i$ lists the elements of $\{x : 1 \leq x \leq n^2\text{ and }x \equiv i \mod n\}$ in decreasing order. Let $G_n$ be the linked chain graph with linking permutation $\pi_n$. Then $(\sigma(G_{n}))_{n \geq 1}$ is unbounded. 
	\end{theorem}
	
	\begin{proof}
		The proof is similar in concept to the one of Theorem~\ref{thm-cc-unbdd-let}: we show that if all $G_n$ have locally semi-consistent partitions with a bounded number $p$ of bags,\footnote{By replacing $p$ by $3p$ if necessary, we may assume that each bag is contained entirely in $A$, $B$ or $C$.} then in fact, they have locally semi-consistent partitions with only 3 bags. In other words, the partition of the linked chain graphs into sets $A, B, C$ should already be a locally semi-consistent partition. This is clearly not the case for $n \geq 2$, as one can find three vertices in $B$ whose $A$-neighbourhoods are in strictly decreasing order, but whose $C$-neighbourhoods are in neither increasing nor decreasing order (and so, no ordering of $B$ will be locally semi-consistent).
		
		\medskip
		
		Suppose now that every $G_n$ has a locally semi-consistent partition into $p$ bags.  Label the vertices of $A$ by $x_1, \dots, x_{n^2}$ in decreasing order of their $B$-neighbourhoods; label the vertices of $B$ by $y_1, \dots, y_{n^2}$ in increasing order of their $A$-neighbourhoods (so that $y_{\pi_n(1)}, \dots, y_{\pi_n(n^2)}$ is decreasing with respect to the neighbourhoods in $C$); finally, label the vertices of $C$ by $z_1, \dots, z_{n^2}$ in increasing order of their $B$-neighbourhoods. 
		
		Let $A_1, \dots, A_{p_1}$, $B_1, \dots, B_{p_2}$, $C_1, \dots, C_{p_3}$ be the parts lying in $A$, $B$ and $C$ respectively (so that $p = p_1 + p_2 + p_3$). We use the labelling described above to obtain, from the partitions of $A, B$ and $C$, three partitions of $[n^2]$. Define:
		
		\begin{itemize}
			\item $A'_i := \{j \in [n^2] : x_j \in A_i\}$;
			\item $B'_i := \{j \in [n^2] : y_j \in B_i\}$;
			\item $C'_i := \{j \in [n^2] : z_{\pi_n^{-1}(j)} \in C_i\}$ (take note of the $\pi_n^{-1}$ in the index).
		\end{itemize}
		
		Now consider the common refinement of the three partitions $(A'_i), (B'_i)$ and $(C'_i)$. This is a partition $D'_1, \dots, D'_r$ of $[n^2]$, with $r \leq p^3$. From this refinement, we construct a new partition $(D_i)$ of $V(G_n)$ by putting $D_i := \{x_j : j \in D'_i\} \cup \{y_j : j \in D'_i\} \cup \{z_{\pi_n^{-1}(j)} : j \in D'_i\}$. By construction, for each $i$, there exist $i_1, i_2, i_3$ such that $D_i \subseteq A_{i_1} \cup B_{i_2} \cup C_{i_3}$. This means that the induced subgraph $G_n[D_i]$ has $\sigma(G_n[D_i]) = 3$, since the partition $(A_i) \cup (B_i) \cup (C_i)$ of $V(G)$ is by assumption locally semi-consistent.
		
		We claim $G_n[D_i]$ is, in fact, a linked chain graph whose linking permutation is the subpattern of $\pi_n$ induced by the indices in $D'_i$. To see this, note that the neighbourhood of $y_j$ in $A$ is the interval $x_1, \dots, x_j$, while its neighbourhood in $C$ is the interval $z_{\pi_n^{-1}(j)}, \dots, z_{n^2}$. Thus by construction, any vertex $y_j \in B \cap D_i$ has its rightmost neighbour from $A$ and its leftmost neighbour from $C$ also in $D_i$. From this, writing $B \cap D_i = \{y_{j_1}, \dots, y_{j_s}\}$ with $j_1 < \dots < j_s$, it is easy to see that this ordering of the vertices in $B \cap D_i$ is strictly increasing with respect to the neighbourhoods in $A \cap D_i$, and the ordering $y_{\pi_n(j_1)}, \dots, y_{\pi_n(j_t)}$ is strictly decreasing with respect to the neighbourhoods in $C \cap D_i$. But $\pi_n$ reorders the $y_{j_l}$ in exactly the same way as the the subpattern of $\pi_n$ induced by the indices in $D_i'$ reorders $\{1, \dots, t\}$, and so our claim follows.
		
		Finally, to arrive at our contradiction, we use Corollary~\ref{cor-vdw} (and the paragraph after it): the $D'_i$ are a partition of $[n^2]$, and thus of the plot of $\pi_n$, into at most $p^3$ parts. It follows that, for any fixed $t$, if $n$ is large enough, one of those parts, say the $i$th, will contain $\pi_t$ as a pattern. By removing the appropriate vertices of $G_n[D_i]$, we can produce a locally semi-consistent partition of $G_t$ into 3 bags, which is not possible for $t \geq 2$, as discussed at the beginning of the proof.
		
	\end{proof}
	
	This construction shows that, as far as local semi-consistency is concerned, things can go wrong even when there is a single bag where two chain graphs meet. A first step to understanding $\sigma$ is an exhaustive analysis of this setting; one would perhaps aim for a clean characterisation of bounded $\sigma$ among linked chain graphs in terms of some conditions on the linking permutations. Since the obstacle to bounded $\sigma$ comes from the co-chromatic number of the permutation graphs, we conjecture (and it should not be overly difficult to prove or disprove):
	
	\begin{conjecture}
		Let $(G_n)_{n \geq 1}$ be a family of linked chain graphs, with linking permutations $\pi_1, \pi_2, \dots$. Let $\mathcal X$ be the hereditary closure of this family. Then $\sigma$ is bounded in $\mathcal X$ if and only if the hereditary closure of the permutation graphs $G_{\pi_i}$ does not contain all unions of cliques, or all complete bipartite graphs. 
	\end{conjecture}
	
	In the general setting of bounded $\gamma$, we would be dealing with several bags $A_1, \dots, A_p$; it would make sense to attempt to generalise the notion of ``linking permutation'' by associating to each bag $A_i$ a set of permutations $\pi^i_{j \to k}$, which intuitively describe how we need to permute the bag $A_i$ to get from an ordering that makes $G[A_i \cup A_j]$ properly ordered to one that makes $G[A_i \cup A_k]$ properly ordered. Of course, there are several complications here: 
	\begin{enumerate}
		\item The chain graphs between pairs of bags will not, in general, be prime, so the permutations $\pi^i_{j \to k}$ will not be uniquely defined. How do we decide which ones to use? 
		
		\item The linking permutations alone do not suffice to characterise boundedness of $\sigma$; one might also need to take into account the interactions between them, much like in the chain circuit setting. 
		
		\item Assuming we obtain a characterisation of bounded $\sigma$ in terms of conditions on the linking permutations and their interactions, how do we transform it into a minimal class characterisation?
		
	\end{enumerate} 
	
	Dealing with those difficulties is a good subject for future research. Finally, before moving on to the final transition, let us relate what we have seen to Problem~\ref{prob:bddlet}. While we have found classes of unbounded $\sigma$ (and thus unbounded lettericity), those classes were not minimal. In fact, with a little care, one can find such a minimal class within our linked chain graph example. To see it, consider the graphs whose linking permutation is direct sum of 21s, or a skew sum of 12s. Note that the corresponding classes are in fact smaller than the conjectured minimal classes of unbounded $\sigma$. The best way to identify the minimal class of unbounded lettericity hiding in those permutations is pictorial, so we will let the figures speak for us.
	
	\bigskip
	
		\begin{figure}[hbt!]
		\begin{subfigure}[t]{1\linewidth}
		\centering
		
		\begin{tikzpicture}[scale = 0.8, transform shape]
			
			\foreach \i in {0,...,7} {
				\foreach \x in {0,...,2} {
					\filldraw (2 * \i, 3 * \x) circle (2pt);
				}
			}
			
			\foreach \i in {0,...,7} {
				\foreach \x in {\i,..., 7} {
					\draw (2 * \i, 6) -- (2* \x, 3);
				}
			}
			
			\foreach \i in {1,...,8}
			\draw (2 * \i - 2, 3) node[above right] {\i}; 
			
			\foreach \i in {0,...,7}
			\draw (12, 3) -- (2 * \i, 0);
			
			\foreach \i in {1,...,7}
			\draw (14, 3) -- (2 * \i, 0);
			
			\foreach \i in {2,...,7}
			\draw (8, 3) -- (2 * \i, 0);
			
			\foreach \i in {3,...,7}
			\draw (10, 3) -- (2 * \i, 0);
			
			\foreach \i in {4,...,7}
			\draw (4, 3) -- (2 * \i, 0);
			
			\foreach \i in {5,...,7}
			\draw (6, 3) -- (2 * \i, 0);
			
			\foreach \i in {6,...,7}
			\draw (0, 3) -- (2 * \i, 0);
			
			\foreach \i in {7,...,7}
			\draw (2, 3) -- (2 * \i, 0);
			
		\end{tikzpicture}
		\caption{A graph with linking permutation 78563412 (a skew sum of 12s)}
	\end{subfigure}
	
			\smallskip
	
	\begin{subfigure}[t]{1\linewidth}
		\centering
		
		\begin{tikzpicture}[scale = 0.8, transform shape]
			
			\foreach \i in {0,...,7} {
				\foreach \x in {0,...,2} {
					\filldraw (2 * \i, 3 * \x) circle (2pt);
				}
			}
			
			\foreach \i in {0,...,7} {
				\foreach \x in {\i,..., 7} {
					\draw (2 * \i, 6) -- (2* \x, 3);
				}
			}
			
			\foreach \i in {1,...,8}
			\draw (2 * \i - 2, 3) node[above right] {\i}; 
			
			\foreach \i in {0,...,7}
			\draw (12, 3) -- (2 * \i, 0);
			
			\foreach \i in {0,...,6}
			\draw (14, 3) -- (2 * \i, 0);
			
			\foreach \i in {0,...,5}
			\draw (8, 3) -- (2 * \i, 0);
			
			\foreach \i in {0,...,4}
			\draw (10, 3) -- (2 * \i, 0);
			
			\foreach \i in {0,...,3}
			\draw (4, 3) -- (2 * \i, 0);
			
			\foreach \i in {0,...,2}
			\draw (6, 3) -- (2 * \i, 0);
			
			\foreach \i in {0,...,1}
			\draw (0, 3) -- (2 * \i, 0);
			
			\foreach \i in {0,...,0}
			\draw (2, 3) -- (2 * \i, 0);
			
		\end{tikzpicture}
		\caption{The same graph with the bottom vertices rearranged}
	\end{subfigure}
	
		\smallskip
	
	\begin{subfigure}[t]{1\linewidth}
		\centering
		
		\begin{tikzpicture}[scale = 0.8, transform shape]
			
			\foreach \i in {0,...,7} {
				\foreach \x in {1} {
					\filldraw (2 * \i, 3 * \x) circle (2pt);
				}
			}
			
			\foreach \i in {1, 3, 5, 7} {
				\foreach \x in {0, 2} {
					\filldraw (2 * \i, 3 * \x) circle (2pt);
				}
			}

			\foreach \i in {1, 3, 5, 7} {
				\foreach \x in {\i,..., 7} {
					\draw (2 * \i, 6) -- (2* \x, 3);
				}
			}
			
			\foreach \i in {1,...,8}
			\draw (2 * \i - 2, 3) node[above right] {\i}; 
			
			\foreach \i in {1, 3, 5, 7}
			\draw (12, 3) -- (2 * \i, 0);
			
			\foreach \i in {1, 3, 5}
			\draw (14, 3) -- (2 * \i, 0);
			
			\foreach \i in {1, 3, 5}
			\draw (8, 3) -- (2 * \i, 0);
			
			\foreach \i in {1, 3}
			\draw (10, 3) -- (2 * \i, 0);
			
			\foreach \i in {1, 3}
			\draw (4, 3) -- (2 * \i, 0);
			
			\foreach \i in {1}
			\draw (6, 3) -- (2 * \i, 0);
			
			\foreach \i in {1}
			\draw (0, 3) -- (2 * \i, 0);
			
			
		\end{tikzpicture}
		\caption{An induced subgraph of the above graph}
	\end{subfigure}
	
		\smallskip
	
	\begin{subfigure}[t]{1\linewidth}
		\centering
		
		\begin{tikzpicture}[scale = 0.8, transform shape]
			
			\foreach \i in {0,...,7} {
				\foreach \x in {1} {
					\filldraw (2 * \i, 3 * \x) circle (2pt);
				}
			}
			
			\foreach \i in {1, 3, 5, 7} {
				\foreach \x in {2} {
					\filldraw (2 * \i, 3 * \x) circle (2pt);
				}
			}
			
			\foreach \i in {0, 2, 4, 6} {
				\foreach \x in {2} {
					\filldraw (2 * \i, 3 * \x) circle (2pt);
				}
			}

			\foreach \i in {1, 3, 5, 7} {
				\foreach \x in {\i,..., 7} {
					\draw (2 * \i, 6) -- (2* \x, 3);
				}
			}
			
			\foreach \i in {1,...,8}
			\draw (2 * \i - 2, 3) node[above right] {\i}; 
			
			\foreach \i in {1, 3, 5, 7}
			\draw (12, 3) -- (2 * \i - 2, 6);
			
			\foreach \i in {1, 3, 5}
			\draw (14, 3) -- (2 * \i - 2, 6);
			
			\foreach \i in {1, 3, 5}
			\draw (8, 3) -- (2 * \i - 2, 6);
			
			\foreach \i in {1, 3}
			\draw (10, 3) -- (2 * \i - 2, 6);
			
			\foreach \i in {1, 3}
			\draw (4, 3) -- (2 * \i - 2, 6);
			
			\foreach \i in {1}
			\draw (6, 3) -- (2 * \i - 2, 6);
			
			\foreach \i in {1}
			\draw (0, 3) -- (2 * \i - 2, 6);
			
			
		\end{tikzpicture}
		\caption{The same induced subgraph, rearranged}
	\end{subfigure}
	
	\caption{A hidden minimal class of unbounded lettericity}
	
	\end{figure}
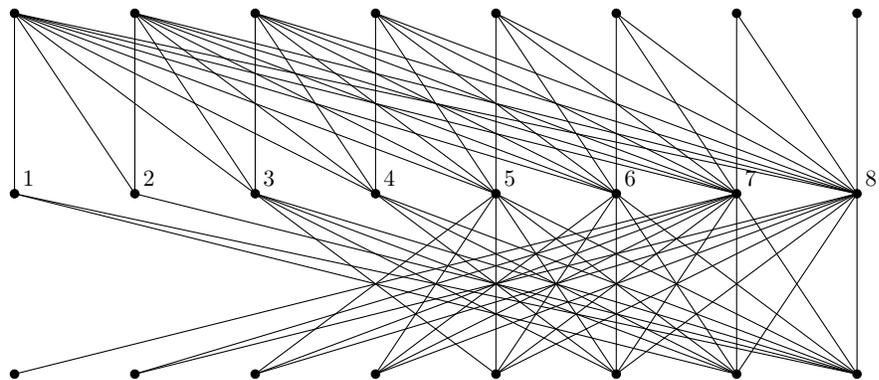
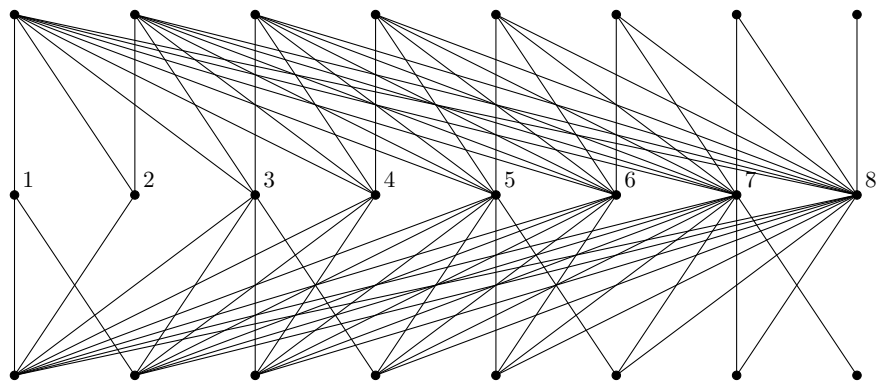
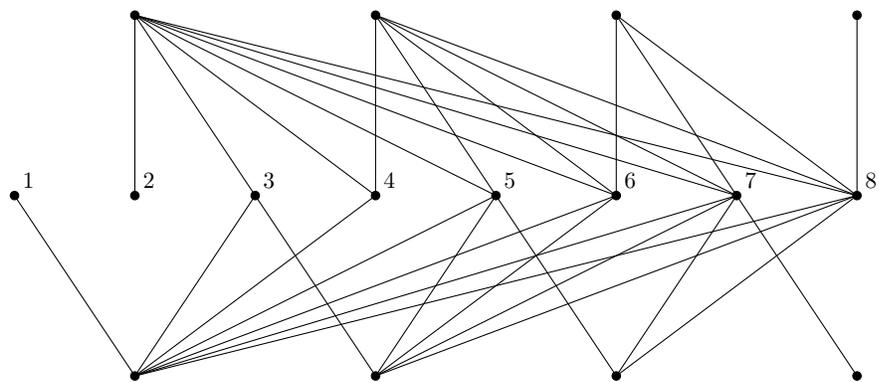
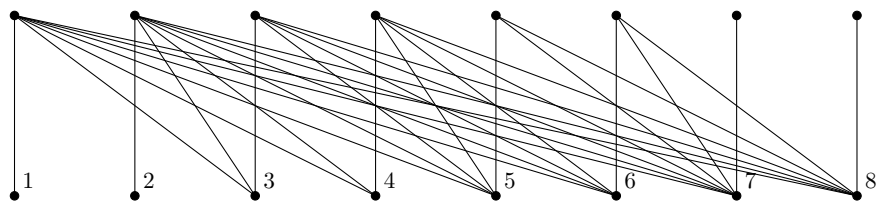

	Lo and behold! It is, up to complementation, the same 4-chain circuit from \cite{alecu-2p3} and \cite{primelet}. The natural follow-up question is: are there obstacles to lettericity that have bounded $\gamma$ and unbounded $\sigma$? Since things get complicated when several bags interact, we will refrain from making a conjecture one way or the other, and instead state it as an open problem:
	
	\begin{problem}
		Are there minimal classes of unbounded lettericity that have bounded $\gamma$ and unbounded $\sigma$? 
	\end{problem}

	\subsection{Between bounded $\sigma$ and bounded $\lambda$}
	
	There is one transition left to examine. We believe this is the simplest of the transitions, but nevertheless, a good understanding of it might give us precious clues on how to proceed with the other ones. It is the transition between monotone griddability, and monotone griddability by a PMM -- in other words, the transition between bounded $\sigma$ and bounded $\lambda$. Constructing an example of a class with bounded $\sigma$ but unbounded $\lambda$ is, at this point, little more than an exercise. We present one in Figure~\ref{fig-lambda-lambda}: we arrange bags in a cycle with properly ordered prime chain graphs between successive bags, except for one pair of bags where we ``twist'' the ordering in one of the bags. That those graphs have unbounded $\lambda$ is essentially a (simpler) variant of Theorem~\ref{thm-gamma-lambda} to show that more bags do not help, together with the remark that, by uniqueness of the proper orderings between bags, we may not make the partition in the figure locally consistent by simply reordering. 
	
	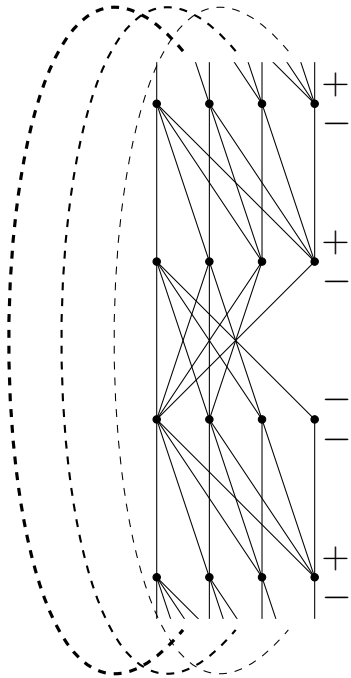
\begin{figure}[ht]
		\centering
		\begin{tikzpicture}[scale = 0.7, transform shape]
			\clip (-4,-2.5) rectangle ++(8.1,14);
			
			\foreach \i in {0,...,3} {
				\foreach \x in {0,...,3} {
					\filldraw (\i, 3 * \x) circle (2pt) node{};
				}
			}
			
			\foreach \i in {-1, 0, 2, 3} {
				\foreach \x in {0,...,3} {
					\foreach \y in {\x, ..., 3} {
						\draw (\y, 3 * \i) -- (\x, 3 * \i + 3);
					}
				}
			}
			
			\foreach \x in {0,...,3} {
				\foreach \y in {\x, ..., 3} {
					\draw (3 - \y, 3) -- (\x, 6);
				}
			}
			
			\filldraw[white] (-0.1, -3.1) rectangle ++(3.2, 2.3);
			\filldraw[white] (-0.1, 12.1) rectangle ++(3.2, -2.3);
			
			\draw (0.5,10) edge [bend right = 140, very thick, looseness = 1.6, dashed] (0.5,-1);
			\draw (1.5,10) edge [bend right = 140, thick, looseness = 1.6, dashed] (1.5,-1);
			\draw (2.5,10) edge [bend right = 140, looseness = 1.6, dashed] (2.5,-1);
			
			\draw (3, 0) node[above right] {\huge $+$}; 
			\draw (3, 0) node[below right] {\huge $-$};
			
			\draw (3, 3) node[above right] {\huge $-$}; 
			\draw (3, 3) node[below right] {\huge $-$};
			
			\draw (3, 6) node[above right] {\huge $+$}; 
			\draw (3, 6) node[below right] {\huge $-$};
			
			\draw (3, 9) node[above right] {\huge $+$}; 
			\draw (3, 9) node[below right] {\huge $-$};

		\end{tikzpicture}   
		\caption{Graphs of bounded $\sigma$ but unbounded $\lambda$}
		\label{fig-lambda-lambda}
	\end{figure}
	
	We have added signs to the figure to emphasize that, like with partial multiplication matrices, there is a parity problem at play. A ``$+$'' indicates that the left-to-right order from the figure gives increasing neighbourhoods in the bag above/below as appropriate, while a ``$-$'' indicates a decreasing order. Complementing the edges between two successive bags flips the two signs ``between'' them, while reversing the order within a bag flips the two signs at that bag. At any rate, those operations cannot change the parity of the number of $-$s. In view of Theorem~\ref{thm-cc-main}, we would not be surprised if, in partial multiplication matrices with cyclic cell graph, those ``twisted cycles in a chain'' are the only obstacles to bounded $\lambda$. We state this as a conjecture:
	
	\begin{conjecture}
		Suppose $\mathcal X$ is a class of permutations monotone griddable by a matrix $M$ with a cyclic cell graph. Then $\mathcal X$ is monotone griddable by a PMM if and only if the corresponding class of permutation graphs does not contain arbitrarily wide constructions like the one in Figure~\ref{fig-lambda-lambda} (or complements).
	\end{conjecture}

	The next step is to try to understand how such a result would generalise; it is likely that the solution is strictly easier than the one to Problem~\ref{lambda-let}, yet very insightful. Finally, in view of Problem~\ref{prob:bddlet}, we have the (slightly weaker) conjecture that the class of graphs like the one depicted in Figure~\ref{fig-lambda-lambda} is {\em minimal} of unbounded lettericity:
	
	\begin{conjecture}
		The class of twisted cycles in a chain is minimal of unbounded lettericity.
	\end{conjecture}
	
	
	
	\section{An alternate terminology}
	\label{sec:loh}
	
	We believe that the problems we have described so far are very attractive. However, it is not entirely clear whether they are formulated ``in the right way''. To elaborate, let us look more closely at the transition between $\lambda$ and lettericity. For one, there is an issue of portability: it is clear that there is a heavy analogy between what happens for graphs, and what happens for permutations. It would be very useful to have a systematic way to transfer results from one world to the other. 
	
	A second issue is that gridding matrices and decoders come with a lot of superfluous information. We almost exclusively care only about bags with non-trivial relationships between them: collinear cells, or pairs of letters with exactly one arc between them in the decoder. We usually do not even care which one of the two arcs appears in the decoder, or whether the row or column's signs are 1 or -1. This information is not useful during most proofs; it is in fact just a burden that needs to be carried around, and that often obfuscates the real intuition behind the arguments. It would be desirable to have a ``clean'' environment to work with, in which this is not an issue.
	
	To tackle those problems, we propose the notion of {\em locally ordered hypergraphs}\footnote{Locally ordered hypergraphs are also discussed briefly in \cite{bddletgg}; they can be used to produce an alternate proof of the main result featured there.} as an attempt to abstract local consistency. We make no claim that this is the best possible way to achieve this -- only that there is a need for a tool that does this. 
	
	\begin{definition}
		A {\em locally ordered hypergraph} (``LOH'' for short) $\mathcal H$ is a hypergraph $(X, E)$ with no isolated vertices, where any hyperedge $e \in E$ has a linear order $\leq_e$ on its elements, which we call the {\em local order} of $e$. The atoms of the algebra generated by $E$\footnote{The algebra generated by $E$ is the smallest family of subsets of $X$ containing all elements of $E$, and closed under complementations, intersections and unions. The atoms are its minimal non-empty elements.} are called the \emph{cells} of the hypergraph, so that $X$ is partitioned by the cells. Those are denoted by $A_{e_{i_1}, e_{i_2}, \dots, e_{i_r}}$, where the $e_{i_\alpha}, \alpha \in [r]$ are the hyperedges containing $A_{e_{i_1}, e_{i_2}, \dots, e_{i_r}}$. In addition, we have the following  
		
		\begin{itemize}
			\item[] {\em local consistency condition}: on any cell $A_{e_{i_1}, \dots, e_{i_r}}$, the linear orders induced by $\leq_{e_{i_1}}, \dots, \leq_{e_{i_r}}$ all agree.
		\end{itemize}
		
	\end{definition}
	
	\begin{definition}
		A LOH $\mathcal H$ has the \emph{global consistency property} (or is \emph{globally consistent}) if there exists a linear order $\leq$ on its vertices which restricts to $\leq_e$ on each hyperedge. 
		
		An alternative way of defining global consistency is via a conflict graph $\conf(H)$: this is the directed graph on $X$, with arcs $(x, y)$ for any elements $x \neq y$ with $x \leq_e y$ for some $e$. Then using topological sorting, $\mathcal H$ is globally consistent if and only if its conflict graph is acyclic.
	\end{definition}

	One can obtain LOHs in straightforward ways from locally consistent chain partitions of graphs, or monotone griddings by PMMs. The main benefit of using them is that they allow one to represent the relevant information about their objects more neatly. In particular, one can produce formulations of what we have shown about the transition between $\lambda$ and lettericity in the language of LOHs. To see how this can be achieved, we need some further definitions.
	
	\begin{definition}
		Let $\mathcal H = (X, E)$ be a LOH with local hyperedge orders $\leq_e$, and let $x \in X$. The {\em split} of $\mathcal H$ at $x$ is the LOH $\mathcal H^x = (X \setminus \{x\}, E')$, where $$E' := \{e \in E: x \notin e\} \cup \{e \cap \{y \in e: y <_e x\}: x \in e\} \cup \{e \cap \{y \in e: y >_e x\}: x \in e\}$$ (we only consider non-empty sets in the union). The local orders on the new hyperedges are inherited from $\mathcal H$. 
	\end{definition}
	
	The intuition behind splitting a LOH at an element $x$ is that we remove all of the comparisons between elements smaller than $x$ and elements larger than $x$. In other words, splitting a LOH brings us closer to global consistency, since we are deleting arcs in the conflict graph. In the monotone gridded setting, it is analogous to adding the horizontal and vertical lines through $x$ to the gridding -- just as the row and column of the gridding to which $x$ belongs get split into two, so do all hyperedges of the LOH containing $x$. Indeed, this is also the role of the red and blue edges from the proof of Theorem~\ref{thm-cc-main}. We can thus define a parameter for LOHs measuring how far they are from being globally consistent:
	
	\begin{definition}
		The {\em global inconsistency} of a LOH is the smallest number of splits needed to make the LOH globally consistent.
	\end{definition}
	
	In this language, Problem~\ref{lambda-let} becomes: what conditions do we need to put on the LOHs in some given class to guarantee that their global inconsistency is bounded? We state this as an open problem, identical in spirit to Problem~\ref{lambda-let}:
	
	\begin{problem}
		What are the obstacles to bounded inconsistency of LOHs?
	\end{problem}
	
	\medskip
	
	We may also formulate an analogue of Theorem~\ref{thm-cc-main} in this language.
	
	\begin{definition}
		A LOH is {\em $k$-cyclic} if its line graph is a cycle on $k$ vertices. 
	\end{definition}
	
	\begin{notation}
		Let $C_{k, l}$ be as in Notation~\ref{not-cycles}, enhanced with an orientation of the edges from $v_{i, j}$ to $v_{i + 1, j'}$ for all appropriate $i, j, j'$.
	\end{notation}
	
	\begin{theorem} \label{thm-cyclic-loh}
		Let $k, l \in \mathbb N$ be fixed. Let $\mathcal X$ be a class of $k$-cyclic LOHs such that their conflict graphs avoid $C_{k, l}$ as an induced subgraph. Then $\mathcal X$ has bounded global inconsistency. 
	\end{theorem}
	
	This statement can be proved in a way completely analogous to Theorem~\ref{thm-cc-main}; we can define a way of gluing together LOHs such that one of them is ``smaller'' than the other -- this allows us to use an induction argument. The red and blue edge construction can then be reproduced by defining a notion of ``infimum'' and ``supremum'' for elements in some hyperedge of the LOH (strictly speaking, we need to be a bit more careful than that and make sure the infimum and supremum are in the appropriate cells). The heart of the argument stays the same, so we skip the details.
	
	\medskip
	
	We believe that the most natural way to state the problems regarding the passage from a locally consistent to a globally consistent regime is in the language of LOHs. In particular, one could expect cleaner descriptions of the obstacles in terms of forbidden structures (perhaps subgraphs that are not necessarily induced) in the conflict graphs of LOHs. To obtain those formulations, it would be interesting to attempt to develop a more robust and general theory of LOHs. 
	
	An interesting final question is whether one can define similar settings to abstract the other transitions at play: 
	
	\begin{problem}
		Develop frameworks analogous to LOHs for the transitions between $\gamma$ and $\sigma$, and between $\sigma$ and $\lambda$.
	\end{problem}

%


\end{document}